\documentclass[11 pt]{amsart}

\usepackage[latin1]{inputenc}
\usepackage{amsfonts,amssymb,amsmath,amsthm}
\usepackage[headinclude,DIV13]{typearea}
\usepackage{hyperref}

\usepackage{geometry}
\usepackage{graphicx, psfrag}

 \usepackage{color}
 \usepackage{setspace}

\newtheorem{rem}{Remark}
\newtheorem{lem}{Lemma}[section]
\newtheorem{pro}{Proposition}[section]
\newtheorem{defi}{Definition}[section]

\newtheorem{theo}{Theorem}

\newtheorem{cor}{Corollary}[section]

\theoremstyle{definition} 
\theoremstyle{definition} 

\renewcommand{\P}{\mathbb{P}}
\newcommand{\R}{\mathbb{R}}

\newcommand{\E}{\mathbb{E}}

\newcommand{\N}{\mathbb{N}}
\newcommand{\Z}{\mathbb{Z}}

\newcommand{\eps}{\varepsilon}

\numberwithin{equation}{section}

\begin{document}

\title[$\Lambda$-Fleming-Viot processes with frequency-dependent selection]
{On $\Lambda$-Fleming-Viot processes with general frequency-dependent selection}

\author{Adrian Gonzalez Casanova}

\address{Universidad Nacional Aut\'onoma de M\'exico, Instituto de Matem\'aticas, \'Area de la Investigaci\'on Cient\'ifica, Circuito Exterior, Ciudad Universitaria, 04510 Coyoacan, CDMX, M\'exico}
\email{adriangcs@matem.unam.mx}

\author{Charline Smadi}
\address{Univ. Grenoble Alpes, INRAE, LESSEM, F-38402 St-Martin-d'Hères, France and 
Univ. Grenoble Alpes, CNRS, Institut Fourier, 38610 Gi\`eres, France}
\email{charline.smadi@inrae.fr}

\maketitle

\begin{abstract}
 We construct a multitype  constant size population model allowing for general selective interactions as 
 well as extreme reproductive events. Our multidimensional model aims for the generality of adaptive dynamics and the tractability of population genetics. 
 It generalizes the idea of Krone and Neuhauser \cite{krone1997ancestral}, and Gonzalez Casanova and Spano \cite{casanova2018duality},
 who represented the selection by allowing 
 individuals to sample several potential parents in the previous generation before choosing the 'strongest' one, by allowing individuals to
 use any rule to choose their parent. The type of the newborn can even not be one of the types of the potential parents, which 
allows modelling mutations. Via a large population limit, we obtain a generalisation of $\Lambda$-Fleming Viot processes, with 
a diffusion term and a general frequency-dependent selection, which allows for non transitive interactions between the different types 
present in the population. We provide some properties of these processes related to extinction and fixation events, and give conditions for them 
to be realised as unique strong solutions of multidimensional stochastic differential equations with jumps.
Finally, we illustrate the generality of our model with applications to some classical biological interactions. 
This framework provides a natural bridge between two of the most prominent modelling frameworks of biological evolution: 
population genetics and eco-evolutionary models.
 \end{abstract}

 \vspace{1cm}
 


\section{Introduction}

Modelling selection falls within the most delicate problems in ecology and evolution. A variety of hypotheses
have been proposed to describe how competing allelic types jostle against each
other in trying to propagate successfully in the next generation 
\cite{ewens2004mathematical,gillespie1984neutral,kimura1968evolutionary,kimura1972population}. 
Despite the complexity of the debate on the concept of selection itself,  in population genetics there
is agreement on the idea that an appropriate measure of the fitness strength
of a given allelic type is related to the probability of its eventual fixation in the population.
This agreement relies on the fact that generally, in population genetics models, 
fitnesses are considered as transitive in the sense that if  in some 
conditions an allele $1$ has a higher fitness than an allele $0$  and an allele $2$ has a higher fitness than an allele $1$, then the allele $2$ will have a higher fitness than the allele $0$.
This assumption impedes the modelling of non transitive interactions, as for instance the well known Rock-Paper-Scissors 
(RPS) interaction, where 
any of the three alleles involved may rise in frequency depending on the frequencies of the two other alleles present 
in the population \cite{sinervo1996rock}.

Eco-evolutionary models aim at taking into account feedbacks between ecological and evolutionary processes.
In this setting, interactions are defined at the individual scale, which allows to model a much larger variety of 
interactions than in population genetics models (as RPS interactions for instance \cite{billiard2017interplay}).
The drawback however, is that the variation of population size makes the analysis more complex.
For instance, even when studying processes faster than the time scale of population size evolution, we have to call upon technical 
results, as large deviations results, to control the population size dynamics (see \cite{champagnat2006microscopic} for a classical example).

The aim of this paper is to construct a constant size population model allowing to consider very general interactions, 
as well as extreme reproductive events.
This is achieved by a generalisation of a construction recently introduced in \cite{casanova2018duality}, based on the concept of ancestral selection 
graph introduced by Krone and Neuhauser \cite{krone1997ancestral}.
The main idea of this construction is that individuals first sample a random number of potential parents from the previous generation, 
and then choose their parent in the previous generation, with a rule which depends on the number and types of sampled potential parents. 
Mathematically, the rule that assigns types to the individuals, \textit{the coloring rule}, 
is a family of random variables which takes values in the type space and is indexed by the types of the potential parents.
Notice that the term 'potential parents' may be misleading as in the current work we 
make it possible for an individual to have a type different from all of its potential parents, this is in order
to model mutations. 
However, to make the link with \cite{casanova2018duality} clear we kept their notations.
Following \cite{casanova2018duality}, we also consider high-fecundity extreme reproductive events ($\Lambda$-events), where 
one individual may give birth to a number of offsprings of order $N$, the size of the population.
After a proper rescaling of time, we derive a large population limit of our population model described by a generator, and 
which can be realised as the strong 
solution of a stochastic differential equation (SDE) under suitable conditions. This SDE generalizes classical Wright-Fisher diffusion with 
selection and $\Lambda$-Wright-Fisher processes to a multidimensional case with general frequency-dependent selection and jumps given 
by a $\Lambda$-measure. We prove general properties on fixation and extinction for this class of models, 
and apply them to classical ecological interactions, as 
RPS interactions or negative frequency-dependent selection.

Notice that in a work conducted in parallel, Cordero et al. \cite{CHS19} also consider a generalisation of \cite{casanova2018duality} to 
a general selection case. Their focus differs substantially from ours, as they consider the $1$-dimensional setting, and concentrate on genealogies and duality 
properties of the frequency process with the ancestral process.
Furthermore, they seek for genealogies that minimise the number of potential ancestors.
We encourage interested readers to also have a look at \cite{CHS19}.

The paper is structured as follows. In Section \ref{sectionmodel} we describe the discrete model. Section \ref{sectiongeneralresults} is dedicated 
to the derivation of the large population limit as well as the statement of general results on the processes under consideration.
In Section \ref{sectionapplications} we present applications of our construction to model a set of ecological interactions, as well as 
properties on the long time behaviour of these specific processes. Finally, the proofs are given in Section \ref{sectionproofs}.

In the sequel, $\N^*:=\{1,2,...,\}$ will denote the set of positive integers, $\N:= \N^* \cup \{0\}$, $\Z := -\N^* \cup \N$, 
and for $N \in \N^*$, $[N]:=\{1,2,...,N\}$.
Finally, $|D|$ will denote the cardinality of a discrete set $D$, and for $K \in \N^*$ and $x \in [0,1]^K $, 
$\|x\|:= x_1+...+x_K$.

\section{The discrete model} \label{sectionmodel}

We consider constant size population models, with size $N \in \N^*$, and with discrete non-overlapping generations.
They generalise the models 
introduced in \cite{casanova2018duality}. 
The main novelty of our approach is that it includes the following features:
\begin{itemize}
 \item We are in a multidimensional setting. We denote by $E= [K]$ the allelic type space, with $K \in \N^*$.
 \item To choose a real parent, knowing the set of potential parents, we will introduce a colouring rule, which may depend on $N$ and on all 
 the characteristics of the set of potential parents (number and frequencies of the different types), and may be random.
 \item The type of the newborn may not be a type carried by one of the potential parents. This allows us to take into account mutations for instance.
\end{itemize}
Multitype models are widespread in the literature, both in population genetics (see for example \cite{Birkner,DM,SW1,SW2}, where even infinite type spaces 
are considered) 
and in eco-evolutionary models (for instance \cite{champagnat2011polymorphic,collet2011quasi,champagnat2014adaptation,bovier2018recovery,bovier2018crossing}). 
There are also interesting instances of colouring (2-types) 
models in which individuals \textit{observe} several potential parents and are coloured according to some rule depending on the observed sample 
\cite{BEK, CHS19, casanova2018duality}. Questions regarding the existence and (lack of) uniqueness of a colouring rule that leads to a prescribed stochastic 
differential equation are addressed in \cite{CHS19}. The literature that accounts for mutations is vast \cite{EG,DSW,BLW}.
However, as far as we are aware of, the family of models introduced in this paper constitutes the first class that integrates 
all the above points, and luckily the flexibility of this family does not compromise its simplicity.

We begin in Section \ref{section_no_extr} with the description of births in generations without extreme reproductive events, 
and describe extreme reproductive events in Section \ref{section_extr}. 
In Section \ref{section_freq}, we introduce the multidimensional frequency process, which will be the process of interest in the rest of the paper.

\subsection{Births without extreme reproductive events} \label{section_no_extr}

In the sequel, $g \in \Z$ will always denote a generation. 
For $ N \in \N^*$, we parametrize the strength of the selection via a probability distribution $Q_N$ on $\N^*$.
It will correspond to a number of 'potential parents' sampled by an individual.
In order to describe the dynamics of our discrete model, we now introduce a graph and the concept of colouring rule:

\begin{defi}\label{GenGra}
Let $N\in\N^*$ and $V:=\Z\times [N]$. Consider a family of independent uniformly on $[N]$ 
distributed random variables 
$$(U_{v,k}, v\in V, k\in\N^*)$$
and a family of independent $Q_N$ distributed random variables 
$$(K_v,v\in V)$$
with values in $\N^*$.  Let $\mathcal{E}$ be the set of directed edges from an individual to its potential parents
$$\mathcal{E}=\{\{v,(g-1,U_{(v,k)})\}, 
\text{ for all }v=(g,l)\in V\text{ such that } 1 \leq l \leq N \text{ and } k\leq K_v\}.$$ \textbf{The genealogical random di-graph} 
with parameters $N$ and $Q_N$ is the random di-graph $(V,\mathcal{E})$.
\end{defi}

For an individual $v$ alive at generation $g$, the variables $(U_{v,k}, k \leq K_v)$ are thus the potential parents of $v$, alive at generation $g-1$.

The idea of this graph is that we assign types to all the vertices 
in some generation, and we want to see types propagating in the subsequent generations. 
To do this we need to specify how a vertex is coloured (receive its type), given the types of the vertices in the previous 
generation which are connected to it. Figure 1 is meant to help understanding the following definition, which at first sight can seem 
technical.

\begin{defi}[Colouring rule] \label{def_col_rule}
Let  
$$\mathcal{C}=\bigcup_{n\in \N^*}E^n $$ 
be the possible sets of potential parents with their type that can be sampled by an individual.
A  $(N,E)$-colouring rule is a family of probability distributions over $E$ $$C_N=\{c^N_{z}(\cdot)\}_{z \in \mathcal{C}}.$$
\end{defi}

The role of the distribution $c^N_{z}$ is to describe how an individual chooses its real parent and its type (in case there is a mutation) 
when its set of potential parents is $z$ and the population size is $N$.

To simplify the notation in some proofs, we will work under the assumption that for any $z \in E$, $ c^N_{z}= \delta_z. $

This condition ensures that if an individual samples only 
one potential parent, it is its real parent. We could omit this assumption but the notations would 
become cumbersome (however, see Remark \ref{rem_colour_rule}).
Notice also that we allow the colouring rule to depend on $N$, which leaves more freedom.

\begin{defi}[Graph colouring without extreme event]\label{forward}
Fix $N\in\N^*$, $Q_N$ a probability distribution on $\N^*$, the space of types $E$ and a $(N,E)$-colouring rule $C_N$. A colouring of the graph $(V,E)$ is a 
function $f:V\rightarrow E$. Let $V_g$ be the vertices in the $g$-th generation. We construct a colouring $f$ recursively 
by first arbitrarily defining $f$ in $V_{0}$ and then extending the colouring to the subsequent generations using 
the colouring rule. Given that the colouring has been constructed in generation $g-1$, 
for any $v\in V_g$, let
$$ z_v=(f(u_1),...,f(u_{K_v}) ),$$
where  $u_k\in V_{g-1}$ is such that $(v,u_k)=(v,U_{v,k})$ for  $k\leq K_v$ (in other words, $u_k$ is a potential parent of $v$).
For every type $i\in E,$
$$
\P(f(v)=i)=c^N_{z_v}(i).
$$
\end{defi}

\begin{center}
\includegraphics[width=.75\textwidth]{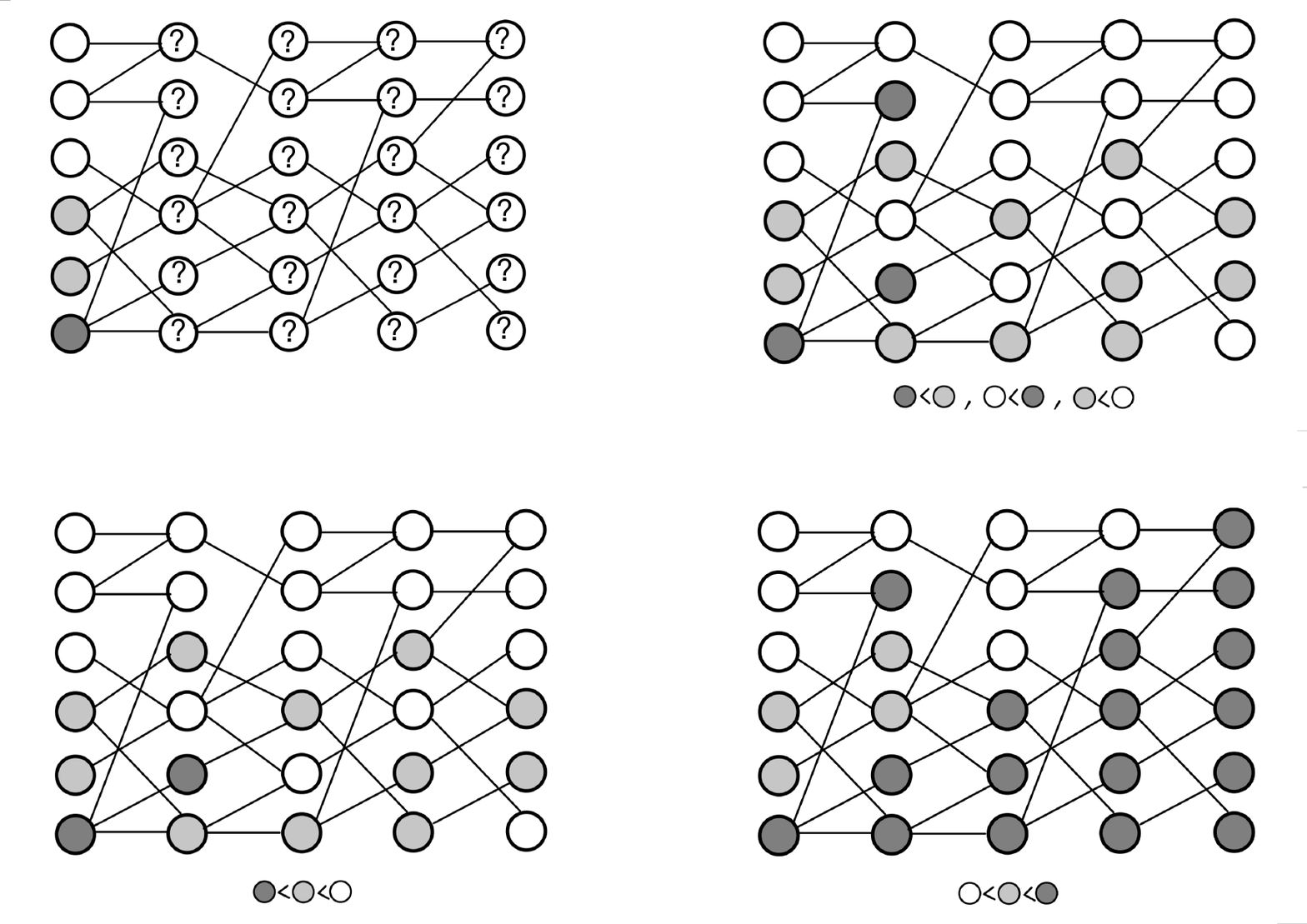}
\end{center}
\textbf{Figure 1}  After the random graph is generated (following the construction in \cite{casanova2018duality}) and the generation zero is arbitrarily 
coloured
 i.e.\ the types of the individuals in that generation are determined, the colour/type of every individual in the subsequent generations is given inductively by 
 a colouring rule. A colouring rule is a \textit{recipe} on how to colour each individual depending  of the types of its observed potential parents. For example, 
 in the last three pictures in this figure we observe the same graph with the same colouring at generation zero. However, the colouring rule is different. Here, in all 
 cases individuals copy their types from their potential parents and the colouring rule determines which colour to use if there are 
 more than one option. 
 We present two pictures with different transitive colouring rules and one with the rock-paper-scissors colouring rule.\\

In the sequel, we will use the terms type and colour interchangeably.
Following the three previous definitions, the reproduction mechanism thus works as follows when there is no extreme reproductive event:

\begin{enumerate}
 \item[(1)] (Choice of potential parents). The individual $v$ of generation $g$ samples 
with replacement
 a number $K_v$ distributed as $Q_N$ of 
 potential parents independently and uniformly 
 distributed on $(g-1) \times [N]$, $((g-1,U_{v,1}),...,(g-1,U_{v,K_v}))$.
 \item[(2)] (Choice of type). The type of $v$ is chosen according to the colouring rule, as described in Definition \ref{forward}.
\end{enumerate}

\subsection{Births with extreme reproductive events} \label{section_extr}

We now want to include the possibility of extreme reproductive events, that is to say events when one individual may generate a non-negligible fraction 
of the next generation.
The motivation to consider such reproductive events comes from the observation that some species, in particular marine species 
\cite{arnason2004mitochondrial,hedgecock1994does} and most viruses \cite{Irwin},
have offsprings distributions with very large variance, and that the Kingman coalescent is not a good model to describe their genealogy 
\cite{hoscheit2018multifurcating}. 
It is instead necessary to include the possibility of merging more than two individuals simultaneously. 
Such kinds of models have seen a growing interest from the discovery of $\Lambda$-coalescent trees by Sagitov \cite{sagitov1999general} 
and Donnelly and Kurtz  \cite{donnelly1999particle} independently, and 
their full description by Pitman \cite{pitman1999coalescents}.
To this purpose, we introduce a random background formed by a sequence of
i.i.d. Bernoulli trials 
$$ H = \{H_g : g \in \Z\} \in \{0, 1\}^\infty,$$
with a probability of success $\gamma_N \in [0, 1]$ and a sequence 
$$Z = \{Z_g : g \in \Z\}$$
of i.i.d. $[0,1]$-valued random elements
with common distribution $\Lambda$. 
$H$ and $Z$ are assumed to be independent,
and we also assume that $\Lambda([0,1])<\infty$.
For every
$g \in \Z$, $H_g = 1$ (resp., $H_g = 0$) indicates that at generation $g$ extreme reproduction
does (resp., does not) occur. $Z$ will give the expected sizes of extreme reproductive events, 
when they occur. 
Let 
$$Y^* = \{Y^*_g: g \in \Z\},$$
be a sequence of i.i.d. random variables, with $Y^*_{g}$ uniformly distributed among individuals in generation $g$ and $Y^*$ independent of $(H, Z)$.
Conditionally on $Y^*_{g-1}$, we thus assume that when $H_g=1$ individuals in generation $g$ choose one parent independently with the law:
\begin{equation} \label{defeta}
\eta_g := \mathcal{B}( Z_g ) \delta_{Y^*_{g-1}} + (1 - \mathcal{B} (Z_g)) U_g ,\end{equation}
where $\mathcal{B}( Z_g )$ is distributed as a Bernoulli random variable with parameter $Z_g$ and $U_g$ is uniformly 
distributed on $[N]$. $\mathcal{B}( Z_g )$ and $U_g$ are independent, and are drawn independently for each individual in 
generation $g$.

\begin{defi}[Graph colouring with extreme event]\label{forward2}
Fix $N\in\N^*$, $Q_N$ a probability distribution on $\N^*$, the space of types $E$ and a finite measure $\Lambda$ on $[0,1]$. 
Given that the colouring has been constructed until generation $g-1$, 
that there is an extreme reproductive event at generation $g$, we have the following colouring for every type $i\in E$ and $v \in V_g$:
$$
P(f(v)=i|Z_g,Y_{g-1}^*)=Z_g \mathbf{1}_{\{i, f(Y^*_{g-1})\}}+(1-Z_g) \frac{ \mathbf{Card} \{ u \in V_{g-1}, f(u)=i \}}{N}.
$$
\end{defi}

Following the previous definitions, the reproduction mechanism thus works as follows when there is an extreme reproductive event:

\begin{enumerate}
 \item[(1)] (Choice of the individual with a large progeny) An individual $Y^*_{g-1}$ is drawn uniformly among individuals of the generation $g-1$. 
 A variable $Z_g$ is also drawn according to the distribution $\Lambda$. It will give the size of the extreme reproductive event.
 \item[(2)] (Choice to be part of the extreme event). The individual $v$ of generation $g$ draws a Bernoulli variable $\mathcal{B}_v(Z_g)$ with parameter $Z_g$.
 \item[(3)] If $\mathcal{B}_v(Z_g)=1$, the parent of $v$ is $Y^*_{g-1}$.
 \item[(4)] If $\mathcal{B}_v(Z_g)=0$, the individual $v$ chooses its parent uniformly in generation $g-1$.
\end{enumerate}

\subsection{Frequency processes} \label{section_freq}

We have described in the two previous sections how individuals choose their parents in the previous generation. It allows us to construct a forward in time 
process, describing the dynamics of the composition of the population in terms of types.
Fix $N\in\N^*$, $Q_N$ a probability distribution on $\N^*$, $\Lambda$ a finite measure on $[0,1]$, 
$\gamma_N \in [0,1]$, and a colouring rule $C_N$. Colour abritrary the graph at generation $0$.
Then colour generations $g \geq 1$ as follows:
\begin{itemize}
\item Colour the graph following Definition \ref{forward} with colouring rule $C_N$ and distribution of
the number of
potential parents $Q_N$ when there is no extreme reproductive event ($H_g=0$ with probability $1-\gamma_N$).
\item Colour the graph following Definition \ref{forward2} with $Z$ $\Lambda$-distributed when there is an extreme reproductive event 
($H_g=1$ with probability $\gamma_N$).
\end{itemize}
Then the frequency process with parameters $(\Lambda, \gamma_N,Q_N,C_N)$ is by definition 
$\bar X^N_g=(X_g^{N,i},i\in E)$ such that
$$
X^{N,i}_g=\frac{1}{N}\big|\{v\in V_{g}: f(v)=i\}\big|.
$$
This Markov chain describes the genetic composition of the population at any generation.
The rest of the paper is devoted to the study of this process.

\section{General results}\label{sectiongeneralresults}

The models and subsequent frequency processes we have just introduced allow us to consider very general interactions.
Indeed we have a total freedoom on the colouring rule, and may even choose a type which is not carried by any of the potential parents sampled by an individual in the spirit of a mutation event.

In this section, we will first present a rescaling of the population process, which generalizes the class 
of Wright-Fisher processes considered until now. 
We will then describe the class of selection functions that we can obtain with our colouring rules, and we will finally 
prove general results on allele extinction and fixation of our class of models.

\subsection{Large population limit}

We will now prove that by rescaling time and taking the large population limit ($N \to \infty$), we may obtain a time and space continuous process. 
We focus on the case 
$$1-Q_N(\{1\})=: \rho_N \to 0, \quad (N \to \infty),$$ 
and to get a proper limit process, we will rescale the time by a factor of order $\min\{\rho_N^{-1},N\}$. For example, in the case $Q_N(\{1\})=1$, 
we obtain the scale of
the neutral Wright Fisher model, and in the case $1-Q_N(\{1\})=O(1/N)$ (corresponding to weak selection), 
the scale is the evolutionary scale $N$. If $1-Q_N(\{1\})=O(1/N^b)$ for $b\in(0,1)$ (which corresponds to moderate selection) the scale is $N^b.$
Before stating our result we need to introduce the following notations, for any finite measure $\Lambda$ on $[0,1]$ and $\alpha \in (0,1/2)$:
\begin{equation} \label{defLamdahat} \hat \Lambda:= \frac{\Lambda}{\Lambda([0,1])}, \quad \Lambda^\alpha_N(z \in .) := \frac{\Lambda(z \in .)}{z^2} \mathbf{1}_{\{z \geq N^{-\alpha}\}}\quad \text{and}
\quad  \hat \Lambda^\alpha_N:= \frac{\Lambda^\alpha_N}{\Lambda^\alpha_N([0,1])}. \end{equation}
Let us also denote for $K \in \N^*$ the $K$-th simplex
$$ \Delta_K:= \{ x \in [0,1]^K, \|x\|\leq 1 \}. $$
Moreover, for $N \in \N^*$, given the colouring rule $C_N$ (recall Definition \ref{def_col_rule}), consider the function $p^N:\Delta_K\to \Delta_K$, $p^N=(p_1^N,...,p_K^N)$  defined by
\begin{equation} \label{defpiN} p_i^N(x):=\P(f(v)=i|v \in V_1,X^N_0=x), \quad \forall \ x \in \Delta_K, i \in E \end{equation}
and $\mu^N:\Delta_K\rightarrow \R^K$, $\mu^N=(\mu_1^N,...,\mu_K^N)$ defined by
\begin{equation} \label{defmuiN} \mu_i^N(x):=p_i^N(x)-x_i, \quad \forall \ x \in \Delta_K, i \in E. \end{equation}
Then we have the following convergence result:

\begin{pro}\label{propconv}
 Let $\Lambda$ be a finite measure on $(0,1]$, $(C_N, N \in \N)$ a sequence of colouring rules and $\alpha$ in $(0,1/2)$.
 Consider for $N \in \N^*$ $(\bar{X}_g^N, g \in \Z)$ the frequency process  with parameters 
 $(\hat{\Lambda}^\alpha_N, \gamma_N,Q_N,C_N)$. Assume that there exist $\kappa>0$ and $\sigma \geq 0$ such that
 \begin{enumerate}
  \item[(i)] $\lim_{N \to \infty} \rho_N^{-1} \mu^N_i(x)= \mu_i(x) \in \R $, $\forall \ i \in E, x \in \Delta_K$,
  \item[(ii)] $\gamma_N = \Lambda^\alpha_N([0,1]) \rho_N/\kappa  + o \left(\Lambda^\alpha_N([0,1]) \rho_N\right)$,
  \item[(iii)] $\lim_{N \to \infty} 1/(N\rho_N)= \sigma/\kappa<\infty$ and $ \rho_N N^{2\alpha} \to 0$,
  \item[(iv)] for $k \geq 2$, $v \in V$, $\lim_{N \to \infty} Q_N(K_{v}=k|K_{v}>1)$ exists and is denoted by $\pi_{k-1}$,
  \item[(v)] for $v \in V$, $\beta:= \lim_{N \to \infty} \E [K_{v}-1|K_{v}>1]= \sum_{k=1}^\infty k\pi_k<\infty$.
 \end{enumerate}
Then, the sequence $(X^N_{\lfloor \kappa t /\rho_N \rfloor}, t \geq 0)$ is tight in the space of càdlàg functions $\mathbb{D}(\R_+,[0,1])$ equipped with the Skorokhod topology. Further, if we denote by $\mathcal{A}^N$ the 
discrete generator of 
$X^N_{\lfloor \kappa t /\rho_N \rfloor}$, the sequence of generators $\mathcal{A}^N$, applied on $f \in \mathcal{C}_2(\Delta_K)\to \R$ and $x \in \Delta_K$, converges to $\mathcal{A}$ given by 
\begin{equation} \label{inf_gene} \mathcal{A}f(x)= \kappa  \sum_{i=1}^K \mu_i(x) \frac{\partial f}{\partial x_i}(x) + 
\frac{\sigma}{2} \sum_{i,j=1}^K \sigma_{ij}(x) \frac{\partial^2 f}{\partial x_i x_j}(x)
+ \sum_{i=1}^K x_i \int_0^1 \left[ f((1-z)x + z \mathbf{e}_i)-f(x) \right]\frac{\Lambda(dz)}{z^2}, \end{equation}
where $\sigma_{ij}(x)=(1_{\{j=i\}}-x_j)x_i$ and $(\mathbf{e}_i, 1 \leq i \leq K)$ is the canonical basis of $\R^K$.
\end{pro}
Assumption (iii) implies that $\min\{\rho_N^{-1},N\}=O(\rho_N^{-1})$. The case $\lim_{N \to \infty} 1/(N\rho_N)=\infty$ is the neutral case, and is well known in the literature.
The first term of the generator describes a general frequency-dependent selective function, the second term is the classical Wright-Fisher diffusion, 
and the last part corresponds to the generator of a $K$-dimensional $\Lambda$-Fleming-Viot process.
Such a general multidimensional process including these three elements has been neither introduced nor studied until now to the best of our knowledge. 
Extending this construction to include multiple simultaneous mergers would make the notation heavier and the proofs even more technical.
For these reasons we do not include a $\Xi$-coalescent part in this work.

\begin{rem} \label{rem_colour_rule}
 Notice that if we do not assume that when only one potential parent is sampled it is necessarily the real parent and its offspring inherits its type
 (see the end of Definition \ref{def_col_rule}), we can obtain a different diffusion term.
More precisely, if we introduce for $i \in E$ and $x \in \Delta_K$,
$$ g_i(x):= \sum_{j=1}^K x_i c_j^N(i), $$
and if $G(x)$ is a random variable with values in $\N^K$, and law given by
$$ \P(G(x)=Y)= {N \choose Y_1,...,Y_K}g_1^{Y_1}(x)...g_K^{Y_K}(x), $$
then for $(i,j) \in E^2$ and if the limit exists,
$$ \sigma \sigma_{ij}(x)= \lim_{N \to \infty} \frac{\kappa}{\rho_N} \E \left[ \left( \frac{G_i(x)}{N} -x_i\right)  \left( \frac{G_j(x)}{N} -x_j\right) \right]. $$
\end{rem}

Under some assumptions on the parameters (see Corollary \ref{cor_LiPu} for examples) we can link the infinitesimal generator 
$\mathcal{A}$ to the solutions to a stochastic differential equation (SDE) generalizing the one dimensional case derived in \cite{casanova2018duality}.

\begin{pro}\label{Generator}
Assume that the following SDE is well posed
\begin{align} \label{SDE_strong}
 dX(t)= &  \mu(X(t))dt +\sqrt{\sigma}\zeta(X(t))dB_t \nonumber \\
 &+ \sum_{i=1}^{K} \mathbf{e}_i \int_0^1 \int_0^1 z 
 \left( \mathbf{1}_{\left\{  \sum_{j=1}^{i-1}X_j(t-) \leq u < \sum_{j=1}^{i}X_j(t-) \right\}}-X_i(t-) \right) 
 \tilde{N}(dt,du,dz),
\end{align}
where $B$ is a $K$-dimensional Brownian motion, $(\mathbf{e}_i, i \leq K)$ is the canonical basis of $E$, $\tilde{N}$ is a compensated Poisson random measure on $\R_+ \times [0,1]
\times [0,1]$ with intensity $dt du \Lambda (dz)/z^2$, and $\zeta$ has the following expression for $x \in \Delta_K$: 
\begin{align} \label{form_zeta} \zeta_{ij}(x)&= 0 \quad \text{if} \quad i<j,\\
\zeta_{ii}(x)&=\sqrt{ \frac{x_i(1-x_1-...-x_i)}{1-x_1-...-x_{i-1}} },\nonumber \\
\zeta_{ij}(x) &= -x_i \sqrt{ \frac{x_j}{(1-x_1...-x_{j-1})(1-x_1...-x_j)} } \quad \text{if} \quad  i > j .\nonumber\end{align}
Then every solution of \eqref{SDE_strong} admits the generator in Equation \eqref{inf_gene}.
\end{pro}

If we aim at showing the weak convergence of the sequence of frequency processes, 
we need to prove existence and uniqueness of the solution to the limiting SDE \eqref{SDE_strong}.

\begin{theo} \label{theo_strong_SDE}
Let us assume that the hypotheses of Proposition \ref{propconv} hold and there exists a unique (in law) solution 
to \eqref{SDE_strong} that we denote by $(X(t), t \geq 0)$. Then the frequency process 
$(X^N_{\lfloor \kappa t /\rho_N \rfloor}, t \geq 0)$ converges in law to the space and time continuous process 
$(X(t), t \geq 0)$, and $(X(t), t \geq 0)$ admits $\mathcal{A}$ defined in \eqref{inf_gene} as infinitesimal generator.
\end{theo}

To our knowledge, similar results have only been derived in the one dimensional case in \cite{casanova2018duality} 
and recently in \cite{CHS19}. 
In this special case (where there are only two alleles) we can ensure the existence of a unique strong solution for the SDE 
under consideration.
We can also provide this property in the multidimensional setting when specific assumptions are fulfilled. 
There had been a lot of recent developments in proving existence and uniqueness of 
multidimensional SDEs (see for example
\cite{BLP, kurtz2007,kurtz2014,li2012strong, xi2017jump}). It should be mention that moment duality is usually useful to prove such results, but we do not have this tool in the general setting that we are investigating. The next corollary presents some sufficient conditions.

\begin{cor} \label{cor_LiPu}
If the hypotheses of Proposition \ref{propconv} hold, the function $\mu$ is continuous, and at least one of the following three statements is true
\begin{enumerate}
\item  $K=2$
\item $\sigma \equiv 0$ and $\int_0^{1}\Lambda(dy)/y<\infty$
\item Selection is transitive (See Section \ref{TO} below)
\end{enumerate}
Then, using the notation of Theorem \ref{theo_strong_SDE},
 $$
(X^N_{\lfloor \kappa t /\rho_N \rfloor}, t \geq 0)\Rightarrow (X(t), t \geq 0),
$$
where $(X(t), t \geq 0)$ is the pathwise unique strong solution to the SDE \eqref{SDE_strong} and admits $\mathcal{A}$, defined in 
\eqref{inf_gene} as its infinitesimal generator.
\end{cor}

The assumption $\sigma \equiv 0$ in statement (2) of the previous result comes from the fact that the diffusion coefficient 
of \eqref{SDE_strong} is not Lipschitz on $\Delta_K$ if $\sigma>0$. However, as long as the solution to \eqref{SDE_strong} is not too close 
to the boundaries of $\Delta_K$ this problem does not arise. Hence, if we introduce the stopping 
times, for any $0<\eps<1/2$:
$$ T_{(\eps,X)}:= \inf \left\{ t \geq 0, \min_{1 \leq i \leq K} X_i(t)\leq \eps \right\}, $$
we can prove the following result:
\begin{lem} \label{unicity_before_Teps}
 Let $0<\eps<1/2$ and assume that 
 $$\int_0^1 \Lambda(dy)/y<\infty.$$
 Then the SDE \eqref{SDE_strong} admits a pathwise unique strong solution $X$ on the time interval $[0,T_{(\eps,X)})$.
\end{lem}

\subsection{Selection functions} \label{section_class_functions}
This section generalizes the ideas of \cite{CHS19} to the multidimensional setting.
As we have a large degree of 
freedom in the choice of the colouring rule, the class of selection functions we can obtain with this construction is very general. 
It is the content of the next  lemma:

\begin{lem} \label{cor1}
Let $N \in \N^*$, $\kappa >0$ and $\mu:\Delta_K\rightarrow \R^K$ 
be a continuous function such that there exists $\lambda>0$, with the property that 
$$g(x):=\frac{\kappa\mu( x)}{\lambda}+x \in \Delta_K, \quad \forall x \in \Delta_K. $$
Assume also that there is a unique (in law) solution to \eqref{SDE_strong}.
Then there exists a sequence of couples of distribution and colouring rule $(Q_N,C_N)$ such that 
the sequence of frequency processes $X^N_{\lfloor \kappa t /\rho_N \rfloor}$ converges to
$(X(t), t \geq 0)$ which has infinitesimal generator $\mathcal{A}$ given by \eqref{inf_gene}.
\end{lem}

The previous Lemma does not provide a colouring rule independent of $N$ leading to $\mu$ as a limit selection function.
This is done in the next lemma, but with a less general class of functions:

\begin{lem}\label{cor2}
Let $\kappa >0$, $n \in \N$ and $\mu:\Delta_K\rightarrow \R^K$ 
be a continuous function such that there exists $\lambda>0$, with the property that  
$$g(x):=\frac{\kappa\mu(x)}{\lambda}+x\in \Delta_K, \quad \forall x \in \Delta_K, $$
and $g$ is a polynomial function 
of degree $n$.  Let for each $i\in E$
$$
g_i( x)=\sum_{z=(z_1,...,z_K)}\binom{n}{z}\alpha^i_{ z}\prod_{r=1}^Kx_r^{z_r}
$$
be the representation of $g_i$ in the Bernstein basis. Then the choices 
$$Q_N=(1-\rho_N)\delta_1+\rho_N\delta_n$$
and 
$$c^N_{ z}(i)=\alpha^i_{ z}$$
for every $N$-dimensional vector $ z$ and $c^N_{ z}=\delta_{ z}$ for every one dimensional vector $z$ lead to the same conclusions 
as in the previous lemma.
\end{lem}

\subsection{On the alleles extinction and fixation} \label{section_alleles_ext}

The end of Section \ref{sectiongeneralresults} deals with extinction and fixation properties of the 
frequency process $(X(t), t \geq 0)$. 
We first provide a sufficient condition for one of the alleles to reach fixation in a finite time.

Before stating it, we introduce a parameter which depends on the mean number of potential parents $\beta$ (defined in Proposition \ref{propconv}) 
and on the measure $\Lambda$:
\begin{equation}\label{def_kappa_star}
 \kappa^* :=\frac{1}{\beta}\int_0^1\log(1-y)\frac{\Lambda(dy)}{y^2}.
\end{equation}

\begin{pro} \label{pro_suff_cond_fix}
 Assume that there are no mutations (the offspring inherits the type of one of the potential parents) and 
 the infinitesimal generator of the frequency process $(X(t), t \geq 0)$ is given by \eqref{inf_gene}. 
 Then there is almost sure fixation 
 of one of the alleles in finite time if at least one of the two following conditions is satisfied:
 \begin{enumerate}
  \item[(i)] $\sigma>0$
  \item[(ii)] $\kappa < \kappa^* <\infty$.
 \end{enumerate}
\end{pro}

\begin{rem}
 Following Definitions \ref{GenGra} and \ref{forward}, the assumption 'there are no mutations' can be written rigorously:
 $$ \P(f(v)=i)=0 \quad \text{if} \quad f(U_{v,k})\neq i \ \forall \ (v,g,k)\quad \text{such that} \quad \{ v,(g-1,U_{v,k}) \} \in \mathcal{E}. $$
\end{rem}

We end this section with a result on the way alleles get extinct in the case of a Wright-Fisher diffusion (without jumps). 
This is an extension of a recent result by Coron and coauthors who proved this result in the neutral case \cite{coron2017perpetual}.

\begin{pro}\label{pro_succ_ext}
 Assume that the SDE \eqref{SDE_strong} is well posed and denote by $(X(t), t \geq 0)$ a solution.
 Suppose that $\Lambda\equiv0$, $\sigma>0$ and that $\mu$ satisfies
 $$ \mu_i(x) = x_i (1-x_i) s_i(x), \quad \text{for all} \ x \in \Delta_K $$
 where $s=(s_1,...,s_K)$ is a continuous function on $\Delta_K$. Then one of the alleles ultimately fixates, and before its fixation, the population experiences successive (and non 
 simultaneous) allele extinctions.
\end{pro}

\begin{rem} \label{rem_xi}
 Notice that as soon as we impose that the type of each individual is among the types of its potential parents (no mutations), $\mu_i(x)$ 
 will be of the form $\mu_i(x)=x_i \tilde{s}_i(x) $, with $\tilde{s}$ bounded.
 Indeed, by definition, for $i \in E$, $N \in \N^*$ and $x \in \Delta_K$,
 $$ \mu_i^N(x)= (Q_N(1)-1)x_i + \sum_{k=2}^\infty Q_N(k) \P(f(v)=i|v \in V_1,X_0^N=x,K_v=k). $$
But, as a type $i$ individual has to be sampled among the potential parents for $v$ to be of type $i$, we get
$$ \P(f(v)=i|v \in V_1,X_0^N=x,K_v=k)\leq 1 - (1-x_i)^k = x_i \left( 1+ (1-x_i)+...+ (1-x_i)^{k-1} \right). $$
Hence in the absence of mutation, to check the conditions of Proposition \ref{pro_succ_ext} on $\mu$, we will just have to prove 
that $\mu_i(x)/(1-x_i)$ is bounded.
 \end{rem}

The next section is dedicated to examples of selection functions we can consider with our models. All of them 
satisfy the assumption of Proposition \ref{pro_succ_ext};

\begin{lem} \label{lem_succ_ext}
If $\mu$ is of the form \eqref{shape_mui}, \eqref{mu_logi}, \eqref{mu_lizard}, \eqref{mu_negfreqdep} or \eqref{mu_posfreqdep}, 
$\Lambda\equiv 0$ and 
the SDE \eqref{SDE_strong} is well posed, then the assumptions of Proposition \ref{pro_succ_ext} are satisfied. 
\end{lem}

\section{Applications} \label{sectionapplications}

In this section, we illustrate the generality of our construction by applying it to model various ecological interactions. 
We also provide, when the calculations are not too cumbersome, the selective function obtained as a limit.

\subsection{Transitive ordering}\label{TO}
It is the classical assumption in population genetics models, consisting in ordering transitively the fitnesses of the different alleles. 
In this scheme, $1$ has a smaller fitness than $\{2,...,K\}$, $2$ has a larger fitness than $1$, but a smaller fitness than $\{3,...,K\}$, and 
so on. 
Thus the colouring rule consists in taking one parent among the potential parents with the highest type, 
$c_z(.)=\delta_{\sup_z}$, where $\sup_z$
is the supreme norm of the vector $z \in E$.

If we consider only two alleles and that the number of potential parents is one or two, we find as a limit the classical Wright-Fisher 
process with selection (plus jumps), where the frequency $(Y_t, t \geq 0)$ of type $1$ individuals is the 
unique strong solution to:
$$ dY_t = -\kappa Y_t(1-Y_t)dt+ \sqrt{\sigma Y_t(1-Y_t)}dW_t+ \int_0^1 \int_0^1 y 
 \left( \mathbf{1}_{\left\{u < Y_{t-} \right\}}-Y_{t-} \right) 
 \tilde{M}(dt,du,dz) $$
where $\tilde{M}$ is a compensated Poisson measure with intensity $ dt du \Lambda (dz)/z^2 $. If we allow the number of potential parents to exceed two, but still with only two alleles, it amounts to the model described and studied in \cite{casanova2018duality}, 
where the selective function is 
\begin{equation} \label{shapemuGCS} \mu_1(x)=\mu_1(x_1,1-x_1)=  \sum_{k=1}^\infty \pi_k \left( x_1^{k+1}-x_1 \right) ,\end{equation}
 where we recall that $\pi$ has been introduced in Proposition \ref{propconv}.

Finally, if we consider a finite number of alleles larger than two, we obtain the following selective function for $i \in E$:
\begin{equation}  \mu_i(x) = \sum_{k=1}^\infty  \pi_k \left((x_0+...+x_i)^{k+1}-(x_0+...+x_{i-1})^{k+1}-x_i\right)\label{shape_mui}. \end{equation}
This can be deduced from \eqref{shapemuGCS}. Indeed for any $i \in E$, if we divide the alleles into two groups: $1$ to $i$ and $i+1$ to $K$, 
the relative frequencies of the two groups evolve as if there were only two types in the population. Doing the same with the groups
$1$ to $i+1$ and $i+2$ to $K$, and so on, allows us to conclude. \\

By extending results in \cite{casanova2018duality} we can 
provide the probability of fixation of the different alleles.
In \cite{casanova2018duality} the authors prove moment duality between the frequency process of the weakest allele and a process $(D_t, t \geq 0)$ named the
ancestral process and corresponding to the limit of the number of potential ancestors of a sample taken at a given generation. 
The ancestral process has the following generator:
\begin{equation} \label{inf_gen_count_proc} Lf(n)=
  \kappa \sum_{i=2}^\infty \pi_i  [f(n+i-1)-f(n)]+ \sigma  {n\choose 2} [f(n-1)-f(n)]+ \sum_{k=2}^n {n \choose k} \lambda_{nk}[f(n-k+1)-f(n)],
\end{equation}
for every $n \in \N$ and $f:\N \to \R$ twice continuously differentiable, where we recall that $\pi$ has been defined in Proposition \ref{propconv},
and for $2 \leq k \leq n$, $\lambda_{nk}$ is given by:
$$ \lambda_{nk}= \int_0^1 y^k (1-y)^{n-k} \frac{\Lambda(dy)}{y^2}. $$
It appears that the behaviour of $(D_t, t \geq 0)$ provides the long time equilibrium of the frequency process $(X(t), t \geq 0)$.
Recall the definition of $\kappa^*$ in \eqref{def_kappa_star}. Then we have the two following possible long time behaviours:

\begin{lem} \label{lem_final_state_transi}
Let us assume that the hypotheses of Proposition \ref{propconv} 
 are satisfied and that $\mu$ is given by \eqref{shape_mui}.
 Denote by $X$ the process with generator $\mathcal{A}$ defined in \eqref{inf_gene}.
 Then
 \begin{enumerate}
  \item[(i)] If $\kappa < \kappa^*<\infty$, $(D_t, t \geq 0)$ has a unique stationary distribution $\nu$ and for $i \in E$ 
  and $x \in \Delta_K$, 
   $$ \P_x\left(\lim_{t \to \infty}X_i(t)=1\right)= \phi_\nu(x_0+...+x_i)-\phi_\nu(x_0+...+x_{i-1}), $$
  where $\phi_\nu$ is the probability generating function of $\nu$.
  \item[(ii)] If $\kappa \geq \kappa^*$, and if we denote by $S$ the maximal label of alleles present at the beginning:
  $$ S:= \sup \{ i \in E, x_i>0\}, $$
  we get
  $$ \P_x\left(\lim_{t \to \infty}X_S(t)=1\right)=1.  $$
 \end{enumerate}
\end{lem}

This result says that if the selection function is not strong enough ($\kappa$ small) or if the variance generated by the choice 
of parents in the previous generation ($\sigma>0$) or extreme 
reproductive events is large enough, the latter can override the selection, and a deleterious allele may fixate.

\begin{rem}
 In the case of transitive fitnesses ($\mu$ given by \eqref{shape_mui}), we could consider more general 
 extreme reproductive events (more precisely $\Xi$ reproductive events, see \cite{casanova2018duality}). We would 
 still obtain the statement (3) of Corollary \ref{cor_LiPu} and Lemma \ref{lem_final_state_transi}.
We do not provide details here for the sake of readability.
 \end{rem}
\begin{rem}
The classic moment duality method that we used here cannot be applied if the coloring rule is not transitive. However, the notion of ancestry can still be useful in more general frameworks as we will see below.
 \end{rem}
\subsection{Transitive ordering with mutations}
We can still consider that the highest type is the best when choosing among potential parents, but add mutations with some probability. 
As a consequence, the offspring can carry a type different from the types of all of its potential parents. 
Notice that this colouring rule is not deterministic. The form of the mutation probability may depend on what we want to model 
(Muller's rachet, fixation of a beneficial mutant,...).

\subsection{Logistic competition}
We want to model competitive interactions, which can depend on the types of the competing individuals.
These kinds of competitive interactions have been widely studied by ecologists and in the setting of varying size stochastic population models,
as they allow to model for instance the preference for different types of resources or spaces depending on the individual's type, 
and lead to non-monotonic dynamics (see \cite{meleard2009trait,champagnat2011polymorphic} for instance).
To model these interactions, we assume that $K_v$ is only supported on $\{1,2\}$, and that the assumptions of Proposition \ref{propconv} are satisfied.
When $K_v=1$ the real parent is the potential parent sampled uniformly at random in the previous generation by $v$. 
When $K_v=2$ and the two potential parents have the same type, the real parent is chosen among them with probability $1/2$ for each. 
Now assume that they have different types that we denote by $i$ and $j$. Then the potential parent of type $i$ 
transmits its type with probability $p_{ij}$, and the potential parent of type $j$ 
transmits its type with probability $p_{ji}$, with $p_{ij}+p_{ji}=1$.
We thus get 
\begin{align} \label{mu_logi}
\mu_i( x)=&  x_i^2+2x_i\sum_{j \neq i} p_{ij}x_j -x_i
= x_i[x_i+2\sum_{j \neq i} (1-p_{ji})x_j -1] \nonumber \\
=&x_i[x_i+2(1-x_i)-2\sum_{j \neq i} p_{ji}x_j -1]= x_i[1-x_i-2\sum_{j \neq i} p_{ji}x_j ].
\end{align}
We thus obtain a competitive Lotka-Volterra function, whose competition coefficients are functions of probabilities $p_{ij}$. 
Notice that the intracompetition coefficient ($1$ here) is constrained by the fact that the frequency of type $i$ cannot exceed 
$1$, and equals $1$ when the population is monomorphic.

\subsection{Rock-Paper-Scissors}\label{sneeeeky}

RPS is a children's game where rock beats scissors, which beat paper,
which in turn beats rock. Such competitive interactions between morphs or species in nature can
lead to cyclical dynamics, and have been documented in various ecological systems 
\cite{buss1979competitive,taylor1990complex,sinervo1996rock,kerr2002local,
kirkup2004antibiotic,cameron2009parasite,nahum2011evolution}. Let us describe two examples of such cycles. 
The first one \cite{sinervo1996rock} is concerned with pattern
of sexual selection in some male lizards. Males are associated to their throat colours, which
have three morphs. Type 1 individuals (orange throat) are polygamous and very aggressive. They
control a large territory. Type 2 individuals (dark-blue throat) are monogamous. They control
a smaller territory. Finally type 3 individuals (prominent yellow stripes on the throat, similar
to receptive females) do not engage in female-guarding behaviour but roam around in search
of sneaky matings. As a consequence of these different strategies, the type 1 outcompetes the
type 2 (because males are more aggressive), which outcompetes the type 3 (as males of type 2
are able to control their small territory and guard their female), which in turn outcompetes the
type 1 (as males of type 1 are not very efficient in defending their territory, having to split their
efforts on several females). The second example \cite{kirkup2004antibiotic} is concerned with the interactions between
three strains of Escherichia coli bacteria. Type 1 individuals release toxic colicin and produce
an immunity protein. Type 2 individuals produce the immunity protein only. Type 3 individuals
produce neither toxin nor immunity. Then type 1 is defeated by type 2 (because of the cost of
toxic colicin production), which is defeated by type 3, (because of the cost of immunity protein
production), which in turn is defeated by type 1 (not protected against toxic colicin).

To model such interactions in our setting, we assume as in the previous example, that $K_v$ is supported on $\{1,2\}$, 
and we introduce the partial order $1<2$, $2<3$, $3<1$. When two potential parents with different types are sampled as potential 
parents, they compete and the 'strongest' transmits its type. We thus define the colouring rule $c^N_z(.) = \delta_{\max \{z_1,z_2\}}$, and obtain
as limit selective function, for $i \in E$ and $x \in \Delta_K$,
\begin{equation} \label{mu_lizard} \mu_i(x)=\kappa x_i (x_{mod_3(i-1)}-x_{mod_3(i+1)}), \end{equation}
where $mod_3(i)$ is the rest of the division of $i$ by $3$.\\

The next result tells that either there is fixation in the RPS selection case, 
or at least the product of allele frequencies approaches $0$ when time increases.

\begin{lem} \label{lem_sneeeeky}
Let us consider the process $(X(t), t \geq 0)$ with infinitesimal generator $\mathcal{A}$ defined in \eqref{inf_gene}
and $\mu$ given by \eqref{mu_lizard}, that is to say the RPS selection.
Then
$$
\E\left[\ln \left(X_1(t) X_2(t) X_3(t)\right)\right] \searrow -\infty, \quad t \to \infty
$$
if and only if $\sigma\neq 0$ or $\Lambda\neq 0$. 
\end{lem}

\begin{rem}
 Notice that if $\sigma$ and $\kappa$ satisfy assumptions of Proposition \ref{pro_suff_cond_fix}, we know that one of the 
 three alleles gets fixed in finite time.
\end{rem}

\subsection{Food web}
The previous example can be generalized to any number of types and any partial order. In particular, it is well suited to represent food webs. 
A possibility could be that $i<j$ indicates that $i$ is eaten by $j$, and in case $i$ and $j$ do not eat each other $i=j$ would 
indicate that we take $p_{ij}=p_{ji}=1/2$ in the example of logistic competition.

\subsection{Negative frequency-dependent selection}

Negative frequency-dependent selection is a common form of selection in nature. It refers to the fact that it is more 
advantageous for an 
individual to be of a type in minority in the population.
In the case of the orchid Dactylorhiza sambucina for instance, there is a negative frequency-dependent selection on colours
due to behavioural responses of pollinators to lack of reward availability \cite{gigord2001negative}. 
In other populations, such a selection can be due to the fact that some predators concentrate on common varieties of prey and overlook rare 
ones \cite{allen1988frequency}.

To represent this form of selection, we sample $K_v$ potential parents as usually, and choose the real parent uniformly at random 
among the potential parents with the less frequent type. 
 If there are several 'less frequent types' we choose one of them uniformly at random.
Notice that to represent such a form of selection, it is necessary to sample more than $2$ potential parents with a positive probability.
Otherwise it would result in no selection ($\mu\equiv 0$).
For the sake of simplicity, to compute $\mu$, we will assume that $K_v$ is supported on $\{1,3\}$, but we could consider any distribution for $K_v$.
In this case, if there are $3$ potential parents, for a type $i$ parent to be chosen there are three possibilities. 
Either, all the potential parents are of type $i$, or only one parent is of type $i$ and the two other parents are
of type $j$ different 
from $i$, or the three parents are of different types, $i$, $j$ and $k$. Summing the probabilities of these 
three events, we get
\begin{align} \label{mu_negfreqdep}
\nonumber \mu_i(x)=& \left[x_i^3 + 3 x_i \left(\sum_{j \neq i} x_j^2 + \frac{1}{3} \sum_{j \neq k, j,k \neq i}x_jx_k\right)-x_i \right]\\
=& x_i
\left[x_i^2-1 + 3 \left( \frac{2}{3} \sum_{j \neq i} x_j^2 + \frac{1}{3}(1-x_i)^2\right) \right]
=2  x_i
\left[ \sum_{j \neq i} x_j^2-x_i(1-x_i) \right].
\end{align}

\subsection{Positive frequency-dependent selection}

Positive frequency-dependent selection refers to the fact that it is more advantageous for an 
individual to be of a type in majority in the population.
This is for instance the fact for warning signal in butterflies, indicating that one individual is poisonous for predators \cite{chouteau2016warning}.
The modelling of this type of selection is similar in spirit to the previous one. 
Among the potential parents, we choose uniformly a parent with the most frequent type. If there are several 'most frequent
types' we choose one of them uniformly at random.
For an easy computation, we make the same assumptions on $K_v$ as in the previous example. 
In this case, if there are $3$ potential parents, for a type $i$ parent to be chosen there are three possibilities. 
Either, all the potential parents are of type $i$, or two parents are of type $i$ and the other one is of type $j$ different 
from $i$, or the three parents are of different types, $i$, $j$ and $k$. Summing the probabilities of these 
three events, we get
\begin{align} \label{mu_posfreqdep}
\mu_i( x)=& x_i
\left[x_i^2-1 + 3 x_i (1-x_i) + \sum_{j \neq k, j,k \neq i}x_jx_k \right]=  x_i
\left[(2x_i-1)(1-x_i) + \sum_{j \neq k, j,k \neq i}x_jx_k \right].
\end{align}

The remainder of the paper is dedicated to the proofs.

\section{Proofs} \label{sectionproofs}

\subsection{Proofs of Section \ref{sectiongeneralresults}}

Before focusing on the convergence of the frequency process, we will first prove Proposition \ref{Generator}, which states that 
the infinitesimal generator $\mathcal{A}$ defined in \eqref{inf_gene} characterizes the solution to \eqref{SDE_strong} when this 
equation is well posed.

\begin{proof}[Proof of Proposition \ref{Generator}] 
We will show that $\mathcal{A}$ is the infinitesimal generator of the process $(X(t), t\geq 0)$.
The drift part (with the selection function $\mu$) is straightforward, and the last part is the classical generator of a $\Lambda$-Fleming Viot 
process (see \cite{griffiths2014lambda} for instance). Let us check that the diffusive part has the good form.
For $i<j \in E^2$, from \eqref{SDE_strong} and Itô's formula, we have:
\begin{align*}
 \frac{d \langle X_i,X_i \rangle}{dt} & = \sigma \sum_{k=1}^i \zeta_{ik}^2 (X)= \sigma\zeta_{ii}^2(X)+  \sigma\sum_{k=1}^{i-1} \zeta_{ik}^2 (X)\\
&  = \sigma \frac{X_i(1-X_1-...-X_i)}{1-X_1-...-X_{i-1}} + \sigma  X_i^2 \sum_{k=1}^{i-1}\frac{X_k}{(1-X_1...-X_{k-1})(1-X_1...-X_k)}\\
&  = \sigma \frac{X_i(1-X_1-...-X_i)}{1-X_1-...-X_{i-1}} + \sigma X_i^2 \sum_{k=1}^{i-1}\left(  \frac{1}{1-X_1...-X_{k}} - \frac{1}{1-X_1...-X_{k-1}}\right)\\
&  = \sigma \frac{X_i(1-X_1-...-X_i)}{1-X_1-...-X_{i-1}} + \sigma X_i^2 \left(  \frac{1}{1-X_1...-X_{i-1}} - 1\right)\\
&  = \sigma \frac{X_i}{1-X_1-...-X_{i-1}}\left( (1-X_1-...-X_i) +  X_i-X_i(1-X_1...-X_{i-1})\right)\\
&  = \sigma X_i(1-X_i),
\end{align*}
and
\begin{align*}
 \frac{d \langle X_i,X_j \rangle}{dt} & = \sigma \sum_{k=1}^i \zeta_{ik} (X) \zeta_{jk} (X)= 
\sigma \zeta_{ii}(X)\zeta_{ji}(X)+ \sigma \sum_{k=1}^{i-1} \zeta_{ik} (X) \zeta_{jk} (X)\\
&  =  -\sigma\sqrt{ \frac{X_i(1-X_1-...-X_i)}{1-X_1-...-X_{i-1}} }X_j \sqrt{ \frac{X_i}{(1-X_1...-X_{i-1})(1-X_1...-X_i)} } \\
& \qquad +  \sigma X_i X_j \sum_{k=1}^{i-1}\frac{X_k}{(1-X_1...-X_{k-1})(1-X_1...-X_k)}\\
&  =  -\sigma\frac{X_i X_j}{1-X_1-...-X_{i-1}} + \sigma  X_i X_j  \left(  \frac{1}{1-X_1...-X_{i-1}} - 1\right)  = - \sigma X_i X_j.
\end{align*}
\end{proof}

\begin{proof}[Proof of Theorem \ref{theo_strong_SDE}]
Under the assumption of existence and uniqueness of the solution of \eqref{SDE_strong}, 
it is enough to prove the $L_1$-convergence of the sequence of generators of $(X_{\lfloor M_t^N\rfloor}^N, t \geq 0)$,
where $M^N$ is a Poisson process with rate $\kappa / \rho_N$, to the generator of $(X(t), t \geq 0)$, evaluated in the $C_2$ functions $[0,1]\mapsto\R$ (see Lemma 19.27 of \cite{kallenberg2006foundations}). 
As $[0,1]$ is compact, it is enough to show pointwise convergence.
Then the claim will follow from Theorems 17.25 and 17.28 of \cite{kallenberg2006foundations}.
Recall the notations of Sections \ref{section_no_extr} and \ref{section_extr}, as well as \eqref{defpiN}-\eqref{defmuiN}.
In order to simplify the computations, we will adopt the following representation for the law of $\bar{X}^N$: 
$$ X_g^N \Big| \left\{ X_{g-1}^N=x \right\} = (1-H_g)\frac{M_x^{(1,N)}}{N} +H_g\frac{M_x^{(2,N)}}{N}, $$
where the laws of $M_x^{(1,N)}$ and $M_x^{(2,N)}$ are defined as follows:
$$ \P\left( M_x^{(1,N)} = y \right) = {N \choose y_1,...,y_K} \left( p_1^N(x) \right)^{y_1}... \left( p_K^N(x) \right)^{y_K},  $$
$$ M_x^{(2,N)}|\{ Z_g=z \} = \sum_{k=1}^N \left\{ (1-B_k)\mathbf{e}_{C_k}+ B_k \mathbf{e}_B \right\} ,$$
where $(y_i, i \leq K) \in \N^K$ such that $\sum_{1 \leq i \leq K}y_i=N$, $(B_k, 1\leq k \leq N)$ are Bernoulli random variables with parameter $z$, for $i \in E$, $1\leq k \leq N$, 
$\P(B=i)=\P(C_k=i)= x_i $, all these variables are independent and we recall that $(\mathbf{e}_i, 1\leq i \leq K)$ is the canonical 
basis of $E$.
Then the discrete generator $\mathcal{A}^N$ of $(X_{\lfloor M_t^N\rfloor}^N)$
applied to any function $f \in C^2 (\Delta_K)$ in a point $x$ of $\Delta_K$ satisfies:
\begin{align*}
 \mathcal{A}^N f(x):&= \lim_{t \to 0} \frac{\E \left[ f\left(X_{\lfloor M_t^N\rfloor}^N\right) \right]-f(x)}{t}\\
 & = \Lambda_N^\alpha([0,1]) \rho_N \kappa^{-1}\frac{\E \left[ f\left(M_x^{(2,N)}/N\right) \right]-f(x)}{\rho_N \kappa^{-1}}\\
 & \quad + \left(1- \Lambda_N^\alpha([0,1]) 
 \rho_N \kappa^{-1} \right) \frac{\E \left[ f\left(M_x^{(1,N)}/N\right) \right]-f(x)}{\rho_N \kappa^{-1}}.\end{align*}
Using the Taylor expansion and the representation of $M_x^{(2,N)}$ gives:
 \begin{align*}
 \mathcal{A}^N f(x) & = \Lambda_N^\alpha([0,1])\int_0^1 
 \E \left[ f\left(\sum_{k=1}^N \left\{ (1-B_k)\mathbf{e}_{C_k}+ B_k \mathbf{e}_B \right\}/N\right) -f(x)\Big|Z=z\right]
 \hat{\Lambda}_N^\alpha(dz)\\
 & \quad + \frac{1- \Lambda_N^\alpha([0,1]) \rho_N \kappa^{-1}}{\rho_N \kappa^{-1}}
 \sum_{i=1}^K \E \left[\left(M_x^{(1,N)}\right)_i/N-x\right] \frac{\partial f}{\partial i}(x)\\
 & \quad + \frac{1- \Lambda_N^\alpha([0,1]) \rho_N \kappa^{-1}}{2\rho_N \kappa^{-1}}
 \sum_{i,j=1}^K \E \left[\left(\left(M_x^{(1,N)}\right)_i/N-x\right)\left(\left(M_x^{(1,N)}\right)_j/N-x\right)\right]
 \frac{\partial^2 f}{\partial ij}(x)+o(1)\\
 &=: A_N+B_N+C_N+ o(1).
\end{align*}

First, from assumption (i) of Proposition \ref{propconv}, we get that
$$ \lim_{N \to \infty} B_N = \kappa \sum_{i=1}^K \mu_i(x) \frac{\partial f}{\partial i}f(x).$$

Second, as $1- Q_N(\{1\}) = o(1)$ when $N$ goes to $+\infty$ it is enough to study $C_N$ to consider the case when individuals sample 
only one potential parent. We get, using assumption (iii) in Proposition \ref{propconv},
$$ \lim_{N \to \infty} C_N = \frac{\sigma}{2} \sum_{i,j=1}^K x_i\left( \mathbf{1}_{i=j}-x_j \right) \frac{\partial^2 f}{\partial ij}f(x).$$

The limit of $A_N$ is more involved to obtain. 
First notice that $M_x^{(2,N)}$ can be rewritten as
$$ M_x^{(2,N)}|\{ Z=z \} = \sum_{k=1}^N \left\{ (1-z)\mathbf{e}_{C_k}+ z \mathbf{e}_B  \right\}
+ \sum_{k=1}^N (z-B_k)(\mathbf{e}_{C_k}-\mathbf{e}_{B}),$$
and that the expectation of the $i$th coordinate ($i \in E$) of the second sum is
$$ \E \left[ \sum_{k=1}^N  (z-B_k)(\mathbf{1}_{C_k=i}-\mathbf{1}_{B=i}) \right]
= \sum_{k=1}^N  \E[z-B_k](x_i-x_i) =0. $$
Hence
\begin{align*}
\E\left[ f\left( M_x^{(2,N)}/N\right)| Z=z \right]& =
\E\left[ f\left(\sum_{k=1}^N \left\{ (1-z)\mathbf{e}_{C_k}+ z \mathbf{e}_B  \right\}/N\right)\right]\\
 &\qquad + \frac{1}{2N^2} \sum_{i,j=1}^K 
\E\left[H_i^NH_j^N
 \frac{\partial^2 f}{\partial ij}(\Xi^N)\right] ,
\end{align*}
where $\Xi^N(z)$ belongs to the segment $[ \sum_{k=1}^N \left\{ (1-z)\mathbf{e}_{C_k}+ z \mathbf{e}_B\right\}/N, M_x^{(2,N)}/N]$ and
$$ H_i^N(z):=\sum_{k=1}^N (z-B_k)(\mathbf{1}_{C_k=i}-\mathbf{1}_{B=i}). $$

But by the Cauchy-Schwarz inequality, we have
\begin{align*} \left| \E\left[H_i^N(z)H_j^N(z) \frac{\partial^2 f}{\partial ij}(\Xi^N(z))\right] \right|
&\leq \sqrt{\E\left[ \left( H_i^N(z) \right)^2 \right]\E\left[ \left( H_j^N(z) \right)^2 \right]} 
\left\| \frac{\partial^2 f}{\partial ij} \right\|_\infty\\
& \leq C \sqrt{\E\left[ \left( H_i^N(z) \right)^2 \right]\E\left[ \left( H_j^N(z) \right)^2 \right]} , \end{align*}
for a finite $C$ where for a function $g$ on $\Delta_K$, $\|g\|_\infty:= \sup_{x \in \Delta_K}|g(x)|$, and we have used that a continuous 
function is bounded on a compact set.
We thus need to bound the term in the square root, which is a simple calculation, as we know explicitely the laws of 
$B$, $(B_k, 1\leq k \leq N)$ and $(C_k, 1 \leq k \leq N)$.
\begin{multline*}
 \E\left[ \left( H_i^N(z) \right)^2 \right] 
 = \E\left[ \left(\sum_{k=1}^N (z-B_k)(\mathbf{1}_{C_k=i}-\mathbf{1}_{B=i}) \right)^2 \right]=\\
  \E\left[ \sum_{k=1}^N (z-B_k)^2(\mathbf{1}_{C_k=i}-\mathbf{1}_{B=i})^2 \right]
  + \E\left[ \sum_{k\neq l=1}^N (z-B_k)(z-B_l)(\mathbf{1}_{C_k=i}-\mathbf{1}_{B=i})
  (\mathbf{1}_{C_l=i}-\mathbf{1}_{B=i}) \right].
\end{multline*}
As the $(B_k, 1\leq k \leq N)$ are mutually independent, with mean $z$, and independent of $B$ and $(C_k, 1 \leq k \leq N)$, 
the last term is null in the previous equality, which implies
$$ \left| \E\left[H_i^N(z)H_j^N(z) \frac{\partial^2 f}{\partial ij}(\Xi^N(z))\right] \right|
\leq C N. $$

In particular this implies that
\begin{align*} \frac{1}{N^2}\int_0^1 \left| \E\left[H_i^N(z)H_j^N(z) \frac{\partial^2 f}{\partial ij}f(\Xi^N(z))\right] \right| \Lambda_N^\alpha(dz)
&\leq \frac{C}{N}\Lambda_N^\alpha([0,1])\\
&\leq C N^{2\alpha -1}\Lambda([0,1]),\end{align*}
by definition of $\Lambda_N^\alpha$ in \eqref{defLamdahat}.
We thus have shown that for large $N$, 
$$ A_n= \int_0^1 
 \left[\E \left[ f\left(\sum_{k=1}^N \left\{ (1-z)\mathbf{e}_{C_k}+ z \mathbf{e}_B \right\}/N\right) \right]-f(x)\right]
 \Lambda_N^\alpha(dz)+o(1) .$$
The next step consists in proving that 
$$ \int_0^1 
 \left[\E \left[ f\left(\sum_{k=1}^N \left\{ (1-z)\mathbf{e}_{C_k}+ z \mathbf{e}_B \right\}/N\right) \right]-
 f\left(\sum_{k=1}^N \left\{ (1-z)x+ z \mathbf{e}_B \right\}/N\right) \right]
 \Lambda_N^\alpha(dz)=o(1). $$
The result is obtained as before by doing a Taylor expansion around $\sum_{1 \leq k\leq N} \left\{ (1-z)\mathbf{e}_{C_k}+ z \mathbf{e}_B \right\}/N$, 
and by introducing 
$$ H_i^N(z):=(1-z)\sum_{k=1}^N (\mathbf{1}_{C_k=i}-x_i). $$
As the calculations are very similar to the previous ones we do not provide details. 
This yields that for large $N$, 
$$ A_N= \int_0^1 
 \left(\E \left[ f\left( (1-z)x+ z \mathbf{e}_B \right) \right]-f(x)\right)
 \Lambda_N^\alpha(dz)+o(1). $$
 To end the proof we need to show that $\Lambda_N^\alpha(dz)$ can be replaced by $\Lambda(dz)/z^2$.
 For this step we will again make use of the Taylor expansion. Indeed we have
 \begin{multline*}
  \E \left[ f\left( (1-z)x+ z \mathbf{e}_B \right)-f(x) \right]\\=
  \sum_{i=1}^K z\E[\mathbf{1}_{B=i}-x_i] \frac{\partial f}{\partial i}(x)
  + \frac{1}{2} \sum_{i,j=1}^K z^2\E\left[(\mathbf{1}_{B=i}-x_i)(\mathbf{1}_{B=j}-x_j)\frac{\partial^2 f}{\partial ij}\left(\Phi^N(z)\right)\right] \\
  = \frac{1}{2} \sum_{i,j=1}^K z^2\E\left[(\mathbf{1}_{B=i}-x_i)(\mathbf{1}_{B=j}-x_j)\frac{\partial^2 f}{\partial ij}\left(\Phi^N(z)\right)\right] ,
 \end{multline*}
where $\Phi^N(z)$ takes values on the interval $[x,x+ z(\mathbf{e}_B-x)]$. Thus there exists a finite constant $C$ such that 
$$ \Big| \E \left[ f\left( (1-z)x+ z \mathbf{e}_B \right)-f(x) \right]\Big|\leq Cz^2. $$
In particular, 
$$
 \left|  \int_0^1 \left(\E \left[ f\left( (1-z)x+ z \mathbf{e}_B \right) \right]-f(x)\right)
 \left(\Lambda_N^\alpha(dz)- \frac{\Lambda(dz)}{z^2}\right) \right|\leq C \int_0^1 \mathbf{1}_{z \leq N^{-\alpha}}\Lambda(dz) 
 \underset{N \to \infty}{\to} 0.
$$
This ends the proof.
\end{proof}

\begin{proof}[Proof of Proposition \ref{propconv}] 
If we do not assume that Equation \eqref{SDE_strong} has a unique strong solution, the convergence of the generators is no longer sufficient to claim the weak 
convergence of the processes. However,  we can show that the sequence of processes is tight using the robust theory introduced in 
\cite{bansaye2018scaling}. 
More precisely, we will apply Theorem 2.3 in \cite{bansaye2018scaling}. To this aim, we need to check that hypotheses \textbf{(H0)} 
and \textbf{(H1)} in in \cite{bansaye2018scaling} are satisfied.

Hypothesis \textbf{(H0)} holds because we are working in a compact space. 
To prove \textbf{(H1)} it is enough to extend the functions and functional spaces in Section 
4 of \cite{bansaye2018scaling} to high dimensions. Following the notation of \cite{bansaye2018scaling} we introduce the function 
$$
h(u_1,u_2,...,u_K)=(1-e^{-u_1},...,1-e^{-u_K})
$$
and the functional space
$$
\mathcal{H}=\{(u_1,....,u_k)\in\Delta_K\rightarrow H_{\bar x}, \bar x\in \R^K\}\text{ with }H_{\bar x}(\bar u)=1-e^{-\sum_{i=1}^Kx_iu_i}.
$$
Then as in \cite{bansaye2018scaling}, \textbf{(H1.1)} holds by construction, \textbf{(H1.2)} follows an application of the Local Stone-Weierestrass Theorem 
(See Appendix 6.4 of 
\cite{bansaye2018scaling} for details). Hypothesis \textbf{(H1.3)} follows from the fact that $\mathcal{H}\subset C_2$ and thus we can apply the uniform convergence 
of the generators that we 
verified in the proof of Theorem \ref{theo_strong_SDE}.
\end{proof}

\begin{proof}[Proof of Corollary \ref{cor_LiPu}]
To prove that the SDE \eqref{SDE_strong} has a unique strong solution, under hypothesis $(1)$ we will apply Theorem 5.1 of \cite{li2012strong}, under  hypothesis $(2)$ we will 
apply Corollary 2.9 in \cite{xi2017jump}, and for hypothesis $(3)$ we will use Lemma 3.6 of \cite{casanova2018duality} and induction.
 
We will first work in the direction of proving the statement under hypothesis $(2)$, and $(1)$ will be obtained after an additional computation.
Consider any colouring rule, as defined in Definition \ref{def_col_rule}. 
Let $i \leq K $.
For any configuration of potential parents with $k_i$ parents of type $i$ ($1\leq i \leq K$), there is 
 a probability $p_i^N(k_1,...,k_{K})$ for the offspring to be of type $i$. Notice that we may have $p_i^N>0$ even if $k_i=0$ 
when we take mutations into account. Moreover, knowing that there are $k_1+...+k_{K}=k$ potential parents, such a configuration 
has a probability 
$$  {k \choose k_1,...,k_{K}}x_1^{k_1}...x_{K}^{k_{K}}$$ 
to be picked. Hence,
$$ \mu_i(x)= \lim_{N \to \infty} \frac{1}{\rho_N} \left( \sum_{k=2}^{\infty} \pi_k^N \left( \sum_{k_1+...+k_{K}=k} {k \choose k_1,...,k_{K}}
 x_1^{k_1}...x_{K}^{k_{K}}p_i^N(k_1,...,k_{K})-x_i\right)\right).  $$
But notice that for any $K$-tuple $(k_1,...,k_K)$ of integers, and $(z,x) \in \Delta_K^2$,
\begin{align*}
 \left| x_1^{k_1}...x_{K}^{k_{K}}-z_1^{k_1}...z_{K}^{k_{K}} \right|
& \leq  \left| x_1^{k_1}-z_1^{k_1}\right|
 \left| x_2^{k_2}...x_{K}^{k_{K}} \right|+z_1^{k_1}
 \left| x_2^{k_2}...x_{K}^{k_{K}}-z_2^{k_2}...z_{K}^{k_{K}} \right|\\
&\leq |x_1-z_1|+ \left| x_2^{k_2}...x_{K}^{k_{K}}-z_2^{k_2}...z_{K}^{k_{K}} \right|
\leq ... \leq \sum_{j=1}^{K}|x_j-z_j|.
\end{align*}
Hence,
\begin{align*}
   &   \left| \mu_i(x)-\mu_i(z) \right|  \leq \\
 &  \lim_{N \to \infty} \frac{1}{\rho_N}
       \left[ \sum_{k=2}^{\infty} \pi_k^N  \left( \left|x_i-z_i\right|+ \sum_{k_1+...+k_{K}=k} {k \choose k_1,...,k_{K}}
p_i^N(k_1,...,k_{K}) \left| x_1^{k_1}...x_{K}^{k_{K}}- z_1^{k_1}...z_{K}^{k_{K}} \right|\right)\right]\\
& = \left|x_i-z_i\right|  +\lim_{N \to \infty} \frac{1}{\rho_N}\left[\sum_{k=2}^{\infty} \pi_k^N  \left( \sum_{k_1+...+k_{K}=k} {k \choose k_1,...,k_{K}}
p_i^N(k_1,...,k_{K}) \left| x_1^{k_1}...x_{K}^{k_{K}}- z_1^{k_1}...z_{K}^{k_{K}} \right|\right)\right]\\
& \qquad \qquad \qquad \qquad \leq C_i \sum_{j=1}^{K}|x_j-z_j|,     \end{align*}
where $C_i$ is a finite constant.
As a consequence, 
\begin{align*}
\left| \langle \mu(x)-\mu(z),x-z\rangle \right| &\leq \sup_{1 \leq i \leq K}C_i \sum_{i,j=1}^{K} |x_i-z_i||x_j-z_j|\\
&\leq  \sup_{1 \leq i \leq K}C_i \sum_{i,j=1}^{K} \left(|x_i-z_i|^2+|x_j-z_j|^2 \right)
\leq  \sup_{1 \leq i \leq K}2C_i K |x-z|^2.
\end{align*}

Now, take $\zeta \equiv 1 $ in Assumption 2.1 and $\rho=C$ in Assumption 2.3 of \cite{xi2017jump}.
We have to check that the following inequalities hold for $(x,z) \in \Delta_K^2$:
 \begin{equation} \label{Ass2.1}
  \langle \mu(x),x\rangle + |\zeta(x)|^2 + \int_0^1 \int_0^1 |c(x,u,y)|^2 du \frac{\Lambda(dy)}{y^2}\leq C (|x|^2+1) ,
 \end{equation}
\begin{equation}\label{Ass2.3.1}
  \langle \mu(x)-\mu(z),x-z\rangle + |\zeta(x)-\zeta(z)|^2 \leq C |x-z|^2 ,
 \end{equation}
 and
\begin{equation}\label{Ass2.3.2}
   \int_0^1 \int_0^1 |c(x,u,y)-c(z,u,y)| du \frac{\Lambda(dy)}{y^2}\leq C |x-z|  ,
 \end{equation} 
where $c_i(x,u,y)= y(\mathbf{1}_{0 \leq u-(x_1+...+x_{i-1})<x_i}-x_i)$ and $C$ is a positive constant, in order to apply 
Corollary 2.9 in \cite{xi2017jump}.
The function $\sigma$ and $\mu$ are bounded. Moreover, 
$$ \int_0^1 \int_0^1 |c(x,u,y)|^2 du \frac{\Lambda(dy)}{y^2}\leq \int_0^1 \int_0^1 du \Lambda(dy)=\int_0^1 \Lambda(dy), $$
which is finite by assumption. Hence \eqref{Ass2.1} holds.
Let us now prove that \eqref{Ass2.3.2} holds. We have
\begin{multline*}
  \int_0^1 \int_0^1 |c_i(x,u,y)-c_i(z,u,y)| du \frac{\Lambda(dy)}{y^2} \\
  =  \int_0^1 \int_0^1 
  \left|(\mathbf{1}_{0 \leq u-(x_1+...+x_{i-1})<x_i}-x_i)-(\mathbf{1}_{0 \leq u-(z_1+...+z_{i-1})<z_i}-z_i)\right| du \frac{\Lambda(dy)}{y}\\
  \leq \left( 2|x_i-z_i|+|x_{i-1}-z_{i-1}| \right)  \int_0^1 \frac{\Lambda(dy)}{y},
\end{multline*}
with the convention $x_{-1}=z_{-1}=0$. Hence, \eqref{Ass2.3.2} holds with $C = 3\int_{(0,1]}\Lambda(dy)/y$. This proves that 
\eqref{SDE_strong} admits a unique strong solution under assumption $(2)$.

To obtain a general result in multiple dimensions we would need to verify that assumption \eqref{Ass2.3.1} holds. Let us focus on $\zeta$ and see what 
the problem is. We have for $i \in E$
\begin{multline*} 
\left|\zeta_{ii}(x) - \zeta_{ii}(z) \right|^2 = 
\left|\sqrt{ \frac{x_i(1-x_1...-x_i)}{1-x_1...-x_{i-1}} } -\sqrt{ \frac{z_i(1-z_1...-z_i)}{1-z_1...-z_{i-1}} } \right|^2\\
\leq \left| \frac{x_i(1-x_1...-x_i)}{1-x_1...-x_{i-1}}-\frac{z_i(1-z_1...-z_i)}{1-z_1...-z_{i-1}}  \right|<|x-z|,
\end{multline*}
where we have used that for $a,b \geq 0$, $|a-b|^2 \leq |a-b||a+b|= |a^2-b^2|$.
For $j > i \in E^2$,
$$ \left|\zeta_{ij}(x) - \zeta_{ij}(z) \right|^2 =0, $$
and for $j < i \in E^2$,
\begin{multline*} 
\left|\zeta_{ij}(x) - \zeta_{ij}(z) \right|^2 = 
\left|x_i \sqrt{ \frac{x_j}{(1-x_1...-x_{j-1})(1-x_1...-x_j)} } -
z_i \sqrt{ \frac{z_j}{(1-z_1...-z_{j-1})(1-z_1...-z_j)} } \right|^2\\
\leq \left| \frac{x_i^2 x_j}{(1-x_1...-x_{j-1})(1-x_1...-x_j)}  -
  \frac{z_i^2z_j}{(1-z_1...-z_{j-1})(1-z_1...-z_j)} \right|<3|x-z|.
\end{multline*}
Unfortunately, it is not enough to apply Corollary 2.9 in \cite{xi2017jump}; we would need $|x-y|^2$. \\

However, it is exactly what one needs to apply Theorem 5.1 of \cite{li2012strong}, which works only in the case $K=2.$ 
To be more precise, to check that the result is true under condition (1), we need to check conditions (2.a), (2.b), (5.a), (5.b) and (5.c) in \cite{li2012strong}.
The calculations are either already done in the previous lines, or very similar. 
We thus do not give the details. Notice that similarly to the case of \cite{casanova2018duality}, the assumption $\int_0^\infty \Lambda(dy)/y$ is not needed 
in this case. The finitness of $\Lambda([0,1])$ is enough.
We conclude that the result follows also under hypothesis $(1)$.\\

To prove statement $(3)$ we will make a change of variable and apply Lemma 3.6 in \cite{casanova2018duality}. Let us choose $i_0 \in [K-1] $ and consider the process 
$$ (Y_{i_0}(t), t \geq 0):= ((X_1+...+X_{i_0})(t), t\geq 0) .$$ 
Adding the $i_0$ equations,
and recalling that \eqref{shape_mui} implies that for $x \in \Delta_K$
$$  \sum_{i=1}^{i_0} \mu_i(x)= \sum_{k=1}^\infty \pi_k \left( \left(x_1+...+x_{i_0}\right)^{k+1}-
\left(x_1+...+x_{i_0}\right)\right), $$
$Y_{i_0}$ should be solution to 
\begin{align*} 
 dY_{i_0}(t)= & \kappa \sum_{k=1}^\infty \pi_k \left( Y_{i_0}^{k+1}(t)-
Y_{i_0}(t)\right)dt + \sqrt{\sigma} \sum_{i=1}^{i_0} \sum_{j=1}^{K}\zeta_{ij}(X(t))dB^{(j)}_t \nonumber \\
 &+ \sum_{i=1}^{i_0} \int_0^1 \int_0^1 z 
 \left( \mathbf{1}_{\left\{  \sum_{j=1}^{i-1}X_j(t-) \leq u < \sum_{j=1}^{i}X_j(t-) \right\}}-X_i(t-) \right) 
 \tilde{N}(dt,du,dz).
\end{align*}
Notice first that the jump term may be reduced to 
$$  \int_0^1 \int_0^1 z 
 \left( \mathbf{1}_{\left\{  u < Y_{i_0}(t-) \right\}}-Y_{i_0}(t-) \right) 
 \tilde{N}(dt,du,dz). $$
Let us now focus on the diffusion term. By definition of $Y_{i_0}$, we get using the calculations derived in the proof of 
Proposition \ref{Generator},
\begin{align*}
 \frac{ d \langle Y_{i_0},Y_{i_0} \rangle}{dt} & =  \frac{d \langle X_1+...+ X_{i_0},X_1+...+X_{i_0} \rangle}{dt} \\
  & = \sum_{i=1}^{i_0}  d \langle X_i,X_i \rangle + \sum_{1 \leq i,j \leq i_0, i \neq j}  d \langle X_i,X_j \rangle 
   =  \sigma\sum_{i=1}^{i_0} X_i(1-X_i) - \sigma\sum_{1 \leq i,j \leq i_0, i \neq j} X_iX_j\\
  & =  \sigma\sum_{i=1}^{i_0} X_i\left(1-X_i - \sum_{1 \leq j \leq i_0, i \neq j} X_j \right)
   = \sigma(X_1+...+X_{i_0})(1-X_1-...-X_{i_0}) \\
  & = \sigma Y_{i_0}\left(1-Y_{i_0}\right) .
\end{align*}
In particular it implies that there exists a Brownian motion $W^{(i_0)}$ such that (see Theorem (4.4) in \cite{durrett2018stochastic} for instance):
$$\sqrt{\sigma} \sum_{i=1}^{i_0} \sum_{j=1}^{K}\zeta_{ij}(X(t))dB^{(j)}_t = \sqrt{ \sigma Y_{i_0}(t)\left( 1 - Y_{i_0}(t) \right)}dW^{(i_0)}_t, $$
and thus $Y_{i_0}$ should be solution to
\begin{align*} 
 dY_{i_0}(t)= & \kappa \sum_{k=1}^\infty \pi_k \left( Y_{i_0}^{k+1}(t)-
Y_{i_0}(t)\right)dt +\sqrt{ \sigma Y_{i_0}(t)\left( 1 - Y_{i_0}(t) \right)}dW^{(i_0)}_t \nonumber \\
 &+  \int_0^1 \int_0^1 z 
 \left( \mathbf{1}_{\left\{  u < Y_{i_0}({t^-}) \right\}}-Y_{i_0}({t^-}) \right) 
 \tilde{N}(dt,du,dz).
\end{align*}
But from Lemma 3.6 in \cite{casanova2018duality} we know that this equation has a unique strong solution. We can make the same reasoning to prove that 
$Y_{i_0-1}$ (if $i_0>1$) is the unique strong solution to
\begin{align*} 
 dY_{i_0-1}(t)= & \kappa \sum_{k=1}^\infty \pi_k \left( Y_{i_0-1}^{k+1}(t)-
Y_{i_0-1}(t)\right)dt +\sqrt{ \sigma Y_{i_0-1}(t)\left( 1 - Y_{i_0-1}(t) \right)}dW^{(i_0-1)}_t \nonumber \\
 &+  \int_0^1 \int_0^1 z 
 \left( \mathbf{1}_{\left\{  u < Y_{i_0-1}(t-) \right\}}-Y_{i_0-1}(t-) \right) 
 \tilde{N}(dt,du,dz),
\end{align*}
where $W^{(i_0-1)}$ is uniquely determined from the Brownian motions $(B^{(j)}, 1 \leq j \leq K)$. As $X_{i_0}= Y_{i_0}-Y_{i_0-1}$, this concludes the proof.
\end{proof}

\begin{proof}[Proof of Lemma \ref{unicity_before_Teps}]
Assume that the hypotheses. of Lemma \ref{unicity_before_Teps} hold.
In the proof of Corollary \ref{cor_LiPu} we saw that the reason for which we could not ensure 
the existence of a strong solution was that all the 
$\zeta_{ij}$'s had not bounded continuous partial derivatives on $\Delta_K$. But for any $0 < \eps < 1/2$,
there exists a finite constant 
$C(\eps)$ such that for $x$ in $(\eps,1-\eps)^K$ and $(i,j,k) \in E^2$,
$$ \left|\frac{\partial \zeta_{ij}(x)}{\partial_k}\right| \leq C(\eps). $$
As a consequence, applying the Taylor Formula we get, for $x$, $z$ in $(\eps,1-\eps)^K$,
$$ \left|\zeta_{ij}(x) - \zeta_{ij}(z) \right|^2 \leq K C(\eps) |x-z|^2. $$
We thus can apply Corollary 2.9 in \cite{xi2017jump} and conclude.
\end{proof}

\subsection{Proofs of general results on the selective functions}

\begin{proof}[Proof of Lemma \ref{cor1}]
For each $N\in\N^*$ consider the random graph with parameters 
$$Q_N=(1-\rho_N)\delta_1+\rho_N\delta_N$$
and colouring rule $c^N_{z}=g(z/N)$ for every $N$-dimensional vector $z$ and $c^N_{z}=\delta_{z}$ for every one dimensional vector $z$. 
Note that this is indeed a colouring rule because we assumed that $g: \Delta_K\mapsto  \Delta_K.$ 
For every $i\in E$, note that 
$$p^N_i(x)=\sum_{z \in E^N} \P_{x}(z_v=z)c^N_{z}(i).$$ 
Observe that $z_v$ is a multinomial random variable with parameters $x$ and $1$ with probability $1-\rho_N$ and is a multinomial random variable with 
parameters $x$ and $N$ with probability $\rho_N$. Hence,
$$
p_i^N(x)=\E[c^N_{z_v}]=(1-\rho_N)x_i+\rho_N\E\left[g_i\left(\frac{z_v}{N}\right)\right] = x_i + \rho_N (g_i(x)-x_i)+ o(\rho_N), \quad (N \to \infty)
$$ 
where in the right hand side we used the Law of Large Numbers. The rest of the proof consists in applying 
Theorem \ref{theo_strong_SDE}.
\end{proof}

\begin{proof}[Proof of Corollary \ref{cor2}]
The proof is similar to the proof of Corollary \ref{cor1} and follows the observation that by construction
$$
\mu_i(x)=g_i(x)-x_i
$$
for all $i \in E$ and $x \in \Delta_K$.
\end{proof}

\subsection{Proofs of general results on alleles extinction and fixation}

\begin{proof}[Proof of Proposition \ref{pro_suff_cond_fix}]
Recall the definition of the ancestral process $(D_t, t \geq 0)$ in \eqref{inf_gen_count_proc}. 
Then if one of the conditions of Proposition \ref{pro_suff_cond_fix} is satisfied, 
the process $(D_t, t \geq 0)$ is positive recurrent (see Theorem 3 of \cite{griffiths2014lambda} and Theorem 1.1 of \cite{Foucart2013lambda}).
In particular, it will reach one in finite time almost surely. This means that all the individuals in the population at a given time 
will have the same ancestor. As we did not allow for mutations in this Proposition, it implies that all the descendants of this 
ancestor (that is to say all the individuals alive at subsequent generations and related to this ancestor by a sequence of edges) will 
have the same type.
\end{proof}
The proof of Proposition \ref{pro_succ_ext} is based on the following two lemmas:

\begin{lem}\label{lemmaZ}
 Assume that the process $Z=(Z_t, t \geq 0)$ satisfies an equation of the form:
\begin{equation}\label{EDSZ}
 dZ_t = (1-Z_t)S(Z_t,t)dt + \sqrt{Z_t(1-Z_t)}dB_t, \quad Z_0 \in (0,1)
\end{equation}
where $B$ is a Brownian motion, and for all $t \geq 0$ and $z \in [0,1]$, $|S(z,t)|\leq C <\infty$ and $S(0,t)=0$.
Then $Z_t \in [0,1]$ for all $t \geq 0$, and
$$ \int_0^{T_1} \frac{1}{1-Z_s}ds = \infty \quad \text{a.s.}, $$
where for $a \in [0,1]$, $ T_a:= \inf \{ t \geq 0, Z_t=a \}. $
\end{lem}

\begin{lem}\label{leminduction}
Let $n \geq 3 $ be in $\N^*$ and $V=(V_1(t),...,V_{n-1}(t), t \geq 0)$ be a process with (possibly inhomogeneous) infinitesimal generator acting 
on $f$ at $v \in \Delta_{n-1} $ of the form
\begin{equation*}  \mathcal{A}_tf(v)= \sum_{i=1}^{n-1} \mu_i(v,t) \frac{\partial f}{\partial v_i}(v)+ 
\sigma \sum_{i,j=1}^{n-1} \sigma_{ij}(v) \frac{\partial^2 f}{\partial v_i v_j}(v), \end{equation*}
where $\mu$ is continuous with respect to $v$ and can be written
$$\mu(v,t)=(v_i(1-v_i)s_i(v,t),1 \leq i \leq n),$$ 
with $\|s\|_\infty\leq C$ for a finite $C$, and 
$$\sigma(v)=\left((1_{\{j=i\}}-v_j)v_i,1 \leq i,j \leq n\right).$$
 Let for all $t \geq 0$, 
 $$1-V_n(t)= V_1(t)+...+V_{n-1}(t),$$ 
 and define the time change $\tau$ on $[0,\infty)$ by 
 $$ \int_0^{\tau(t)}\frac{1}{1-V_n(s)}ds=t, \quad \forall \ t \geq 0. $$
 Next let us introduce the process
\begin{align*} Y=&\left(Y_1(t),...,Y_{n-2}(t), 1- Y_1(t)-...-Y_{n-1}(t),t \geq 0\right) 
\\:=& \left( \frac{V_1}{1-V_n}(\tau(t)),..., \frac{V_{n-2}}{1-V_n}
 (\tau(t)), \frac{V_{n-1}}{1-V_n}(\tau(t)) ,t \geq 0\right) .\end{align*}
 Then the stochastic process $Y$ has a (possibly inhomogeneous) infinitesimal generator acting 
on $f$ at $y \in \Delta_{n-1} $ of the form
\begin{equation*}  \mathcal{\tilde{A}}_tf(y)= \sum_{i=1}^n \tilde{\mu}_i(y,t) \frac{\partial f}{\partial y_i}(y) + 
\sigma \sum_{i,j=1}^n \sigma_{ij}(y) \frac{\partial^2 f}{\partial y_i y_j}(y), \end{equation*}
where $\tilde{\mu}$ is continuous with respect to $v$ and can be written
$$\tilde{\mu}(y,t)=(y_i(1-y_i)\tilde{s}_i(y,t),1 \leq i \leq n-1),$$ 
with $\|\tilde{s}\|_\infty\leq C'$ for a finite $C'$, and 
$$\sigma(y)=((1_{\{j=i\}}-y_j)y_i,1 \leq i,j \leq n-1).$$
 \end{lem}

Before proving these two Lemmas, we prove Proposition \ref{pro_succ_ext}.

\begin{proof}[Proof of Proposition \ref{pro_succ_ext}]
 The fact that one of the alleles ultimately fixates is a consequence of Proposition \ref{pro_suff_cond_fix}.
If $n=2$, the result is immediate. Hence, we assume that $n \geq 3$.
From Lemma \ref{lemmaZ},
$$ \int_0^{T_1^{V_n}}\frac{1}{1-V_n(s)}ds=\infty, $$
where $T_1^{V_n}$ is the hitting time of $1$ by the process $V_n$.
Indeed, 
$$ \left| - \sum_{i=1}^{n-1} \mu_i(V(t),t) \right| \leq \sum_{i=1}^{n-1} V_i(t)\left(1-V_i(t)\right) |s_i(V(t),t)|
\leq C \sum_{i=1}^{n-1} V_i(t) = C \left( 1-V_n(t) \right).$$
 We thus may introduce the time change $\tau$ on $[0,\infty)$ such that 
 $$ \int_0^{\tau(t)}\frac{1}{1-V_n(s)}ds=t, \quad \forall \ t \geq 0. $$
As for any $t <\infty $, $V_n(\tau(t))<1$, we may consider the process
$$ Y(t)=(Y_1(t),...,Y_{n-2}(t)):= \left( \frac{V_1(\tau(t))}{1-V_n(\tau(t))},...,
\frac{V_{n-2}(\tau(t))}{1-V_n(\tau(t))} \right). $$
Thanks to Lemma \ref{leminduction}, we know that the stochastic process $(Y_1(t),...,Y_{n-2}(t))_{t \geq 0}$ 
has a (possibly inhomogenous) infinitesimal generator acting 
on $f$ at $y \in \Delta_{n-2} $ of the form
\begin{equation*}  \mathcal{\tilde{A}}_tf(y)= \sum_{i=1}^{n-2} \tilde{\mu}_i(y,t) \frac{\partial f}{\partial y_i}(y) + 
\sigma \sum_{i,j=1}^{n-2} \sigma_{ij}(y) \frac{\partial^2 f}{\partial y_i y_j}(y), \end{equation*}
where 
$$\tilde{\mu}(y,t)=(y_i(1-y_i)\tilde{s}_i(y,t),1 \leq i \leq n-2),$$ 
$|\tilde{s}|\leq C'$ for a finite $C'$, and 
$$\sigma(y)=((1_{\{j=i\}}-y_j)y_i,1 \leq i,j \leq n-2).$$
We end the proof following the proof of Theorem 2.1 in \cite{coron2017perpetual}.
By the induction hypothesis, the process $Y$ undergoes $n-2$ successive extinctions at times 
$$ E_1^{Y}<...<E_{n-2}^{Y}<\infty. $$
Hence, 
$$ \tau\left(E_1^{Y}\right)<...<\tau\left(E_{n-2}^{Y}\right)<\tau\left(\infty\right)=T_1^{V_n}. $$
if $\{ T_1^{V_n}< \infty \}$, the times $ \tau\left(E_1^{Y}\right)$, ..., $\tau\left(E_{n-2}^{Y}\right)$, $T_1^{V_n}$ correspond 
to the $n-1$ extinction times of alleles $\{1,2,...,n-1\}$. This concludes the proof, as from the first part of the theorem 
we know that 
$$\P(\bigcup_{i=1}^n \{T_1^{V_i}<\infty\})=1.$$
\end{proof}

To prove Lemma \ref{lemmaZ}, we cannot use general results of \cite{coron2017perpetual} because of the selection term.
We will instead apply Itô's formula to an auxiliary function of the process $Z$.

\begin{proof}[Proof of Lemma \ref{lemmaZ}]
First notice that $Z$ is continuous and that the states $0$ and $1$ are absorbing.
This implies that $Z$ stays in the interval $[0,1]$.
Let $\eps,Z_0>0$ be such that $Z_0<1-\eps$.
By applying Itô's Formula to the function $f(x):=- \ln (1-x)$ we get for any positive $t$:
\begin{multline*}
 - \ln ( 1- Z_{t \wedge T_{1-\eps}} )+ \ln (1-Z_0) \\ =\int_0^{t \wedge T_{1-\eps}} S(Z_s,s)ds + \int_0^{t \wedge T_{1-\eps}} \sqrt{\frac{Z_s}{1-Z_s}}dB_s
 + \frac{1}{2} \int_0^{t \wedge T_{1-\eps}}\frac{Z_s}{1-Z_s}ds.
\end{multline*}
We have
$$ \P\left( \lim_{\eps \to 0^+} \left\{- \ln ( 1- Z_{T_{1-\eps}} )+ \ln (1-Z_0)\right\}= \infty, T_1 < \infty \right)= 
\P(T_1 < \infty), $$
and 
$$ \P\left( \limsup_{\eps \to 0^+} \left|\int_0^{T_{1-\eps}} S(Z_s,s)ds\right|< \infty,
T_1 < \infty \right)= \P(T_1 < \infty), $$
as $S$ is bounded by assumption.
We deduce
\begin{equation} \label{infinite_sum} \P\left( \lim_{\eps \to 0^+} \left\{ 
\int_0^{T_{1-\eps}} \sqrt{\frac{Z_s}{1-Z_s}}dB_s + \int_0^{T_{1-\eps}}\frac{Z_s}{1-Z_s}ds
\right\}= \infty, T_1 < \infty \right)= \P(T_1 < \infty).\end{equation}
We will now prove that \eqref{infinite_sum} implies the following property:
\begin{equation} \label{infinite_term} \P\left( 
\int_0^{T_{1}}\frac{Z_s}{1-Z_s}ds
= \infty, T_1 < \infty \right)= \P(T_1 < \infty).\end{equation}
The random variable 
$$ \int_0^{T_{1-\eps}}\frac{Z_s}{1-Z_s}ds $$
is non negative and non increasing with $\eps$. As a consequence, it has a nonnegative limit when $\eps$ goes to $0$, which 
can be finite or infinite. 
Let us consider a measurable event $A$ such that 
\begin{equation} \label{var_quad_finite}  \lim_{\eps \to 0^+} \left\{ 
\int_0^{T_{1-\eps}}\frac{Z_s}{1-Z_s}ds
\right\}=
\int_0^{T_1}\frac{Z_s}{1-Z_s}ds< \infty \quad \text{a.s. on } A \cap \{T_1 < \infty\}. \end{equation}
Then from \eqref{infinite_sum}, we get that
\begin{equation} \label{mart_infinite}  \lim_{\eps \to 0^+} \left\{ 
\int_0^{T_{1-\eps}}\sqrt{\frac{Z_s}{1-Z_s}}dB_s
\right\}= \infty \quad \text{a.s. on } A \cap \{T_1 < \infty\}. \end{equation}

Let us introduce the process 
$$ M^{(\eps)}_t:=\int_0^{t \wedge T_{1-\eps}} \sqrt{\frac{Z_s}{1-Z_s}}dB_s. $$
$ M^{(\eps)}$ is a continuous martingale. In particular, it is a time change of a Brownian motion, and there exists a Brownian 
motion $W$ such that (see Theorem (4.4) in \cite{durrett2018stochastic} for instance):
$$ M^{(\eps)}_t = W_{\int_0^{t \wedge T_{1-\eps}}\frac{Z_s}{1-Z_s}ds}. $$
This implies:
\begin{align*}
 \E \left[ \mathbf{1}_{A \cap \{T_1 < \infty\}} e^{- \int_0^{t \wedge T_{1-\eps}} \sqrt{\frac{Z_s}{1-Z_s}}dB_s} \right]&=
 \E \left[ \mathbf{1}_{A \cap \{T_1 < \infty\}} e^{- W_{\int_0^{t \wedge T_{1-\eps}}\frac{Z_s}{1-Z_s}ds}} \right]\\
 & \geq \E \left[ \mathbf{1}_{A \cap \{T_1 < \infty\}} e^{- \sup \{ W_u,u \leq \int_0^{t \wedge T_{1-\eps}}\frac{Z_s}{1-Z_s}ds\}} \right].
\end{align*}
As $\{T_1 < \infty\}$ implies $\{T_{1-\eps} < \infty\}$, we may let $t$ go to infinity, and obtain
\begin{align} \label{before_lim}
\nonumber \E \left[ \mathbf{1}_{A \cap \{T_1 < \infty\}} e^{- \int_0^{T_{1-\eps}} \sqrt{\frac{Z_s}{1-Z_s}}dB_s} \right]
 & \geq \E \left[ \mathbf{1}_{A \cap \{T_1 < \infty\}} e^{- \sup \{ W_u,u \leq \int_0^{T_{1-\eps}}\frac{Z_s}{1-Z_s}ds\}} \right]\\
 & \geq \E \left[ \mathbf{1}_{A \cap \{T_1 < \infty\}} e^{- \sup \{ W_u,u \leq \int_0^{T_1}\frac{Z_s}{1-Z_s}ds\}} \right].
\end{align}
But from \eqref{mart_infinite}, we get that 
$$ \lim_{\eps \to 0}\E \left[ \mathbf{1}_{A \cap \{T_1 < \infty\}} e^{- \int_0^{T_{1-\eps}} \sqrt{\frac{Z_s}{1-Z_s}}dB_s} \right]=0, $$
and \eqref{var_quad_finite} implies that the right hand side of \eqref{before_lim} is positive if and only if 
the event $A \cap \{T_1 < \infty\}$ has a positive probability. We thus deduce that 
$$ \P \left( A \cap \{T_1 < \infty\}\right)=0, $$
which implies \eqref{infinite_term}.
We conclude the proof of Lemma \ref{lemmaZ} by noticing that
$$ \P\left( \int_0^{T_{1}}\frac{1}{1-Z_s}ds
= \infty, T_1 = \infty \right)= \P(T_1 = \infty).$$
\end{proof} 
 
 \begin{proof}[Proof of Lemma \ref{leminduction}]
Let us denote by $\tilde{\mathcal{L}}$ the infinitesimal generator of the process
  $$ \left( \frac{V_1}{1-V_n}(t),..., \frac{V_{n-2}}{1-V_n}(t), 1-V_n(t) \right)_{t \geq 0}. $$
For any real valued function $f$ defined on $\{ (y_1,...,y_{n-2},1-x_n) \in\Delta_{n-2}\times \Delta_1 \}$, 
twice differentiable,
we can write for $x_n \neq 1$,
$$ \tilde{\mathcal{L}}f (y_1,...,y_{n-2},1-x_n) = :\mathcal{L}(f \circ g)(x_1,...,x_{n-1}), $$
where by definition, for $(x_1,...,x_{n-1}) \in [0,1]^{n-1}$ such that $0<x_1+...+x_{n-1}\leq 1$,
$$ y=(y_1,...,y_{n-2},1-x_n) = g(x_1,...,x_{n-1})= \left( \frac{x_1}{x_1+...+x_{n-1}}, ..., \frac{x_{n-2}}{x_1+...+x_{n-1}}, x_1+...+x_{n-1} \right).  $$
It can be verified that
\begin{multline*}
 \mathcal{L}f(x)= \mathcal{L}f(x_1,...x_{n-1}):= \sum_{i=1}^{n-1}\mu_i(x)\frac{\partial f}{\partial x_i}(x)
 + \sum_{i=1}^{n-1}x_i(1-x_i)\frac{\partial^2 f}{\partial x_i^2}(x)
 - \sum_{i\neq j=1}^{n-1}x_ix_j\frac{\partial^2 f}{\partial x_i x_j}(x).
\end{multline*}
The calculations for the diffusion part of $\tilde{\mathcal{L}}$ have been done in \cite{coron2017perpetual}. 
We thus only need to compute the term of drift.
\begin{align*} A(x):=& \sum_{i=1}^{n-1} \mu_i(x) \frac{ \partial }{\partial x_i}(f\circ g (x))\\
= &  \sum_{i=1}^{n-2} \mu_i(x) \left(\sum_{j=1}^{n-2} \frac{ \partial y_j }{\partial x_i} \frac{\partial f}{\partial y_j} 
( y) 
+ \frac{ \partial (1-x_n) }{\partial x_i} \frac{\partial f}{\partial (1-x_n)} ( y)\right)\\
& \quad + \mu_{n-1}(x)\left(
\sum_{j=1}^{n-2} \frac{ \partial y_j }{\partial (1-x_n)} \frac{\partial f}{\partial y_j} 
( y) + \frac{\partial f}{\partial (1-x_n)} ( y)
\right)
\\
= &  \sum_{i=1}^{n-2} \mu_i(x) \left(\sum_{j=1, j \neq i}^{n-2} \frac{- x_j }{(1-x_n)^2} \frac{\partial f}{\partial y_j} 
(y) +  \frac{(1-x_n)-x_i}{(1-x_n)^2}  \frac{\partial f}{\partial y_i} (y)   + \frac{\partial f}{\partial (1-x_n)} 
( y)\right)\\
& \quad + \mu_{n-1}(x)\left(
\sum_{j=1}^{n-2} \frac{ -x_j }{(1-x_n)^2} \frac{\partial f}{\partial y_j} 
(y) + \frac{\partial f}{\partial (1-x_n)} (y)
\right).
\end{align*}
Rearranging the terms, we get:
$$ A(x) = \frac{1}{1-x_n}\left( \sum_{i=1}^{n-2}  \left( \mu_i(x) - y_i \sum_{j=1}^{n-1}\mu_j(x) \right) \frac{\partial f}{\partial y_i}(y)
\right)+ \sum_{j=1}^{n-1}\mu_j(x) \frac{\partial f}{\partial (1-x_n)} ( y). $$
If we introduce the notation 
$$ \tilde{\mu}_i(y,x)=  \mu_i(x) - y_i \sum_{j=1}^{n-1}\mu_j(x), $$
we get, adding the diffusive part:
\begin{align*}
 \tilde{\mathcal{L}}(y)=& \frac{1}{1-x_n}  \sum_{i=1}^{n-2} \left( \tilde{\mu}_i(y,x) \frac{\partial f}{\partial y_i}(y)
+ \sigma y_i (1- y_i)\frac{\partial^2 f}{\partial y_i^2}(y)- \sum_{j = 1, j \neq i}^{n-2}y_i  y_j\frac{\partial^2 f}{\partial y_i \partial y_j}(y)\right)\\
&+ \sum_{i=1}^{n-1}\mu_i(x) \frac{\partial f}{\partial (1-x_n)} ( y) + \sigma x_n(1-x_n) \frac{\partial^2 f}{\partial (1-x_n)^2}(y).
\end{align*}
This proves that the diffusive part of the process $Y$ is the same as the diffusive part of a $(n-2)$-dimensional 
Wright-Fisher process.
We still have to prove that $\tilde{\mu}$ satisfies the assumptions we want.
$\tilde{\mu}$ can be rewritten in two different ways. Firstly we have
\begin{align} \label{tildexi} \tilde{\mu}_i(y,x)
& = x_i(1-x_i) s_i(x) - y_i \sum_{j=1}^{n-1}\mu_j(x) = y_i \left( (1-x_n)(1-x_i)s_i(x)- \sum_{j=1}^{n-1}\mu_j(x) \right). \end{align}
Secondly,
\begin{align} \label{1tildexi} \tilde{\mu}_i(y,x)
& = (1-y_i) \mu_i(x) - y_i \sum_{j=1, j \neq i}^{n-1}\mu_j(x) 
 = (1-y_i) \mu_i(x) - y_i \sum_{j=1, j \neq i}^{n-1}x_j(1-x_j)s_j(x).\end{align}
Let us focus on the last term. By the assumption, we know that there exists a finite $C$ such 
that for any $x \in \Delta_{n-1}$, 
and for $1 \leq j \leq n-1$,
$$ \left|  y_i (1-x_j)s_j(x) \right| \leq C .$$
Hence, we get
\begin{align} \label{1tildexi2}
\left|y_i \sum_{j=1, j \neq i}^{n-1}x_j(1-x_j)s_j(x) \right| & \leq C
 \sum_{j=1, j \neq i}^{n-1}x_j = C(1- x_i-x_n) = C(1-x_n)(1-y_i).
\end{align}
Thus, from \eqref{tildexi}, we deduce that there exists a finite $C'$ such that
$$ \limsup_{y \in [0,1]^{n-1}, x \in \Delta_{n-1},  y_i \to 0^+} \frac{|\tilde{\mu}_i(y,x)|}{y_i(1-y_i)}<C', $$
and from \eqref{1tildexi} and \eqref{1tildexi2} we deduce that
$$ \limsup_{y \in [0,1]^{n-1}, x \in \Delta_{n-1},y_i \to 1^-} \frac{|\tilde{\mu}_i(y,x)|}{y_i(1-y_i)}<C'. $$
As we are working on compact sets, 
we obtain that for $1 \leq i \leq n-2$, $\tilde{\mu}_i(y,x)/(y_i(1-y_i))$ is continuous with respect to $y$.
This concludes the proof.
\end{proof}

\subsection{Proofs of results on specific examples}

\begin{proof}[Proof of Lemma \ref{lem_succ_ext}]
 From Remark \ref{rem_xi}, we know that it is enough to check that in the examples under consideration, $\mu_i(x)/(1-x_i)$ is bounded, for 
 $i \in E$ and $x \in \Delta_K$.\\

\noindent \textbf{Transitive ordering case}: Recall that according to \eqref{shape_mui},
$$ \mu_i(x) = \sum_{j=1}^\infty \pi_j \left((x_0+...+x_i)^{j+1}-(x_0+...+x_{i-1})^{j+1}-x_i\right).$$
We can rewrite $\mu_i$ as 
the sum of two functions $\alpha_i$ and $\beta_i$ as follows
 \begin{align*}
\mu_i(x) & =  \sum_{j=1}^\infty \pi_j 
\left(x_i\sum_{k=0}^j (x_0+...+x_i)^{k}(x_0+...+x_{i-1})^{j-k} -x_i\right) \\
& = x_i\sum_{j=1}^\infty \pi_j 
\left(\sum_{k=0}^j (x_0+...+x_i)^{k}(x_0+...+x_{i-1})^{j-k} -1\right) \\
& = \underbrace{ x_i \sum_{j=1}^\infty \pi_j 
((x_0+...+x_i)^{j}-1)}_{\alpha_i}\\
&+\underbrace{x_i (x_0+...+x_{i-1})\sum_{j=1}^\infty \pi_j\sum_{k=0}^{j-1} (x_0+...+x_i)^{k}(x_0+...+x_{i-1})^{j-k-1}}_{\beta_i}.\end{align*}
Firstly notice that
 \begin{align*} |\alpha_i|&=  x_i \sum_{j=1}^\infty \pi_j 
\left(1-(x_0+...+x_i)^{j}\right)\\& = 
  x_i (1-(x_0+...+x_i)) \sum_{j=1}^\infty \pi_j 
\left(\sum_{k=0}^{j-1}(x_0+...+x_i)^{k}\right)\\& \leq 
 x_i (1-x_i) \sum_{j=1}^\infty j \pi_j =  x_i (1-x_i) \beta,
\end{align*}
where we recall that $\beta$ has been defined in (v) of Proposition \ref{propconv}.
Secondly, we have,
 \begin{align*} |\beta_i|&=
  x_i (x_0+...+x_{i-1})\sum_{j=1}^\infty \pi_j 
\left(\sum_{k=0}^{j-1} (x_0+...+x_i)^{k}(x_0+...+x_{i-1})^{j-k-1}\right)\\&\leq
  x_i (1-x_i)\sum_{j=1}^\infty j\pi_j =  x_i (1-x_i) \beta.
 \end{align*}
As a consequence, 
$$ |\mu_i(x)|\leq 2 \kappa  \beta x_i(1-x_i). $$

\vspace{.4cm}

\noindent \textbf{RPS or food web case}: For $ i \in E$,
  $$ |\mu_i(x)|= x_i\left| \sum_{ j \neq i, j<i}x_j-\sum_{ j \neq i, i<j}x_j  \right|
  \leq  x_i\left( \sum_{ j \neq i, j<i}x_j+\sum_{ j \neq i, i<j}x_j  \right) \leq 2 x_i (1-x_i). $$

\vspace{.4cm}

\noindent \textbf{Negative frequency-dependent selection}:
Recall that when the distribution of $K_v$ is concentrated on $\{1,3\}$, we get
$$ \mu_i(x)=2 x_i
\left[ \sum_{j \neq i} x_j^2-x_i(1-x_i) \right]. $$
As 
$$ \sum_{j \neq i} x_j^2\leq \sum_{j \neq i} x_j = 1-x_i, $$
the assumptions of Proposition \ref{pro_succ_ext} are satisfied. If $K_v=p \notin \{1,3\}$, for a parent of type $i$ to be chosen, 
a potential parent of type $i$ has to be present. This ensures that $\mu_i(x)$ can be written as $x_i \tilde{s}_i(x)$ with 
$\tilde{s}$ bounded. Moreover, if we exclude the case when there are only parents of type $i$, 
which contributes with a term $\rho_Nx_i(x_i^{p-1}-1)$ in $\mu_i$, 
we need to choose $u<p$ parents of type $i$, which occurs with probability $x_i^u$, and $p-u$ parents with a type different from $i$, 
which occurs with probability $(1-x_i)^{p-u}$.
We thus obtain a finite sum of terms smaller than 
$$x_i^u(1-x_i)^{p-u} \leq 1-x_i.$$
This ensures that the negative frequency-dependent selection rule satisfies the assumptions of Proposition \ref{pro_succ_ext}.

\vspace{.4cm}
\noindent \textbf{Positive frequency-dependent selection} and \textbf{Logistic competition}:
In this cases, the calculations are very similar to the previous case. We thus do not give details.

%
%
%
\end{proof}

\begin{proof}[Proof of Lemma \ref{lem_final_state_transi}]
Lemma \ref{lem_final_state_transi} is a consequence of Lemma 4.7 in \cite{casanova2018duality}.
Let us first assume condition of (i) of Lemma \ref{lem_final_state_transi}. Then from Lemma 4.7 in \cite{casanova2018duality},
$(D_t, t \geq 0)$ has a unique stationary distribution.
Let us choose $i \in E$ and divide $E$ into two subsets, $E_1:=\{0,...,i\}$ and $E_2:=\{i+1,...,K\}$. 
 Treating the types of $E_1$ as the weak type 
 $0$ in \cite{casanova2018duality}, and the types of $E_1$ as the selected type 
 $1$ in \cite{casanova2018duality}, we also get applying this lemma that
 $$ \P_x\left(\lim_{t \to \infty}(X_0(t)+...+X_i(t))=1\right)= \phi_\nu(x_0+...+x_i) $$
 and
 $$ \P_x\left(\lim_{t \to \infty}(X_0(t)+...+X_i(t))=0\right)= 1-\phi_\nu(x_0+...+x_i). $$
 Applying the same trick to $E_1:=\{0,...,i-1\}$ and $E_2:=\{i,...,K\}$ allows to conclude the proof of point (i).
  Let us now assume the condition of (ii). Then the process $X_S$ has the same properties that the process $X_1$ in \cite{casanova2018duality}, 
 for which there is almost sure fixation in finite time. This ends the proof.
\end{proof}

 \begin{proof}[Proof of Lemma \ref{sneeeeky}]
The proof will be based on studying the action of the generator $\mathcal{A}$ of $X$ over the Logarithm. 
In other words, we will use the Logarithm as a Lyapunov function to study the long term behaviour of $(X_t, t \geq 0)$. 
First, applying Theorem 1 of \cite{griffiths2014lambda}, we obtain the existence of random variables $V$ and $W$ in $[0,1]$ with continuous densities such that
the infinitesimal generator $\mathcal{A}$ of $(X(t), t \geq 0)$ applied to a function $g$ at $x$ on $\Delta_3$ can be rewritten
\begin{align*} \label{inf_gene} \mathcal{A}g(x)&= \sum_{i=1}^3 \mu_i(x) \frac{\partial g}{\partial x_i}(x) + 
\sum_{i,j=1}^3 \sigma^2_{ij}(x) \left( \sigma  \frac{\partial^2 g}{\partial x_i x_j}(x)+ \frac{1}{2}
\E\left[ \frac{\partial^2 g}{\partial x_i x_j}(x(1-W)+VW\mathbf{e}_i) \right] \right)\\
&= \sum_{i=1}^3 \mu_i(x) \frac{\partial g}{\partial x_i}(x) + 
\sum_{i=1}^3 \left(\sum_{j=1}^3 \sigma^2_{ij}(x) \left( \sigma  \frac{\partial^2 g}{\partial x_i x_j}(x)+ \frac{1}{2}
\E\left[ \frac{\partial^2 g}{\partial x_i x_j}(x(1-W)+VW\mathbf{e}_i) \right] \right)\right). \end{align*}

Let 
$$f_i(X(t))=\ln(X_i(t))=\ln\left(1-X_{mod_3(i-1)}(t)+X_{mod_3(i+1)}(t)\right).$$ 
Then, a direct calculation leads to
$$ \sum_{i=1}^3 \mu_i(x) \frac{\partial f_1}{\partial x_i}(x)= 2 \kappa (x_3-x_2).  $$
Moreover, for any $y \in \Delta_3$,
\begin{align*}
 \sum_{j=1}^3 \sigma^2_{1j}(x)  \frac{\partial^2 f_1}{\partial x_i x_j}(y)
 = x_1(1-x_1)\left(-\frac{1}{y_1^2}\right)-x_1x_2\left(\frac{1}{y_1^2}\right)- x_1x_3\left(\frac{1}{y_1^2}\right)
 = - 2\frac{x_1(x_2+x_3)}{y_1^2},
\end{align*}
\begin{align*}
 \sum_{j=1}^3 \sigma^2_{2j}(x)  \frac{\partial^2 f_1}{\partial x_i x_j}(y)
 =-x_1x_2\left(\frac{1}{y_1^2}\right)+ x_2(1-x_2)\left(-\frac{1}{y_1^2}\right)- x_2x_3\left(-\frac{1}{y_1^2}\right)
 = - 2\frac{x_1x_2}{y_1^2},
\end{align*}
and
\begin{align*}
 \sum_{j=1}^3 \sigma^2_{3j}(x)  \frac{\partial^2 f_1}{\partial x_i x_j}(y)
 = - 2\frac{x_1x_3}{y_1^2},
\end{align*}
by a similar computation.
Adding all the terms yields for the function $f_1$:
$$  \mathcal{A}f_1(x)= 2 \kappa (x_3-x_2)- (1-x_1) \left( \frac{4\sigma}{x_1}+ \E \left[ \frac{x_1}{x_1(1-W)+VW}+ \frac{1}{1-W} \right) \right],  $$
and adding over the three functions yields
\begin{align*}  \mathcal{A}(f_1+f_2+f_3)(x)= - \sum_{i=1}^3(1-x_i) \left( \frac{4\sigma}{x_i}+ \E \left[ \frac{x_i}{x_i(1-W)+VW}+ \frac{1}{1-W} \right) \right] 
 \leq -C, \end{align*}
where $C$ is a positive constant if and only if $\sigma\neq 0$ or $\Lambda\neq 0$. We can now conclude using the generator equation that
\begin{multline*}
\lim_{t\rightarrow \infty} \E[\ln(X_1(t)X_2(t)X_3(t))]-\ln(x_1x_2x_3)
\\=\lim_{t\rightarrow \infty} \int_0^t\E_x[Af_1(X(s))+Af_2(X(s))+Af_3(X(s))]ds
\\\leq-\lim_{t\rightarrow \infty} \int_0^t C ds=-\infty .
\end{multline*}
\end{proof}

\section*{acknowledgments}The authors are grateful to S. Billiard, A. Bovier, M. E. Caballero, F. Cordero, S. Hummel and 
E. Schertzer for many interesting discussions.
We thank the CNRS for its financial support through its competitive funding programs on interdisciplinary research.
AGC was supported by CONACYT through the project Ciencia Basica A1-S-14615, and
CS by the Chair "Modélisation Mathématique et Biodiversité" of VEOLIA-Ecole Polytechnique-MNHN-F.X.

\end{document}